\documentclass[11pt,oneside, a4paper,reqno]{amsart}
\usepackage[utf8]{inputenc}
\usepackage[T1]{fontenc}
\usepackage{lmodern}
\usepackage{etoolbox}

\usepackage{geometry}
\usepackage[all]{xy}
\pdfpagewidth=\paperwidth
\pdfpageheight=\paperheight
\usepackage{xcolor}

\usepackage{amsmath,amsfonts,amssymb,mathtools,manfnt,amsthm}

\newtheorem{thm}{Theorem}[section]
\newtheorem{cor}[thm]{Corollary}
\newtheorem{lem}[thm]{Lemma}
\newtheorem{prop}[thm]{Proposition}

\theoremstyle{definition}
\newtheorem{dfn}[thm]{Definition}

\newtheorem{ntn}[thm]{Notation}

\theoremstyle{remark}
\newtheorem{rmk}[thm]{Remark}

\newtheorem{example}[thm]{Example}

\usepackage{tikz}
\usepackage{manfnt}
\usetikzlibrary{matrix}
\usetikzlibrary{arrows}
\usetikzlibrary{positioning}
\usetikzlibrary{calc}
\usepackage{tikz-cd}

\usepackage{enumitem}

\def\namedlabel#1#2{\begingroup
	#2%
	\def\@currentlabel{#2}%
	\phantomsection\label{#1}\endgroup
}
\makeatother

\makeatletter
\g@addto@macro\bfseries{\boldmath}
\makeatother

\usepackage{datetime}
\usepackage{url}

\DeclareMathOperator{\diag}{diag}

\DeclareMathOperator{\id}{id}

\DeclareMathOperator{\spaan}{span}

\DeclareMathOperator{\Aut}{Aut}
\DeclareMathOperator{\End}{End}
\DeclareMathOperator{\dom}{dom}

\DeclareMathOperator{\coker}{coker}
\DeclareMathOperator{\stab}{stab}
\DeclareMathOperator{\sgn}{sgn}
\DeclareMathOperator{\ox}{\otimes}

\newcommand\restr[2]{{
		\left.\kern-\nulldelimiterspace 
		#1 
		\vphantom{|} 
		\right|_{#2} 
}}

\newcommand{\ol}[1]{\overline{#1}}

\newcommand{\pointme}[4]{\chi^{#1}_{#2,#3,#4}}
\newcommand{\chare}[5]{\chi^{#1}_{#2,#3,#4,#5}}

\newcommand{\pointmv}[4]{\epsilon^{#1}_{#2,#3,#4}}

\newcommand{\us}[2]{u_{s(#1),\alpha_{\ol{#1}}(#2)}}

\newcommand{\bs}{\backslash}

\newcommand{\CC}{\mathbb{C}}

\newcommand{\NN}{\mathbb{N}}

\newcommand{\TT}{\mathbb{T}}

\newcommand{\ZZ}{\mathbb{Z}}

\newcommand{\Aa}{\mathcal{A}}

\newcommand{\Cc}{\mathcal{C}}

\newcommand{\Ee}{\mathcal{E}}

\newcommand{\Gg}{\mathcal{G}}

\newcommand{\Kk}{\mathcal{K}}

\newcommand{\Oo}{\mathcal{O}}

\newcommand{\Tt}{\mathcal{T}}

\newcommand{\Zz}{\mathcal{Z}}

\newcommand{\hl}[1]{{\em#1}}

\begin{document}

\title{A Cuntz-Pimsner model for the $C^*$-algebra of a graph of groups }
\author{Alexander Mundey}
\author{Adam Rennie}


\date{}

\begin{abstract}
	We provide a Cuntz-Pimsner model for graph of groups $C^*$-algebras.
	This allows us to compute the $K$-theory of a range of examples and show that graph of groups $C^*$-algebras can be realised as Exel-Pardo algebras. We also make a preliminary investigation of whether the crossed product algebra of  Baumslag-Solitar groups acting on the boundary of certain trees satisfies Poincar\'e duality in $KK$-theory. By constructing a $K$-theory duality class we compute the $K$-homology of these crossed products.
\end{abstract}

\maketitle


\section{Introduction}

In this article we provide a Cuntz-Pimsner model for the $C^*$-algebras of graphs of groups introduced in \cite{BMPST17}. By associating an operator algebra
to a graph of groups, we obtain invariants
of the dynamics via the $K$-theory and $K$-homology of the algebra.

Graphs of groups were introduced by Bass \cite{Bas93} and Serre \cite{Ser80} to study the structure of groups via their action on trees. They have since become an important tool in geometric group theory and low-dimensional topology, see eg. \cite{DD89}.

The relationship between graphs of groups and associated operator algebras has developed sporadically. A construction which parallels ours, for graphs of groups of finite type, was considered by Okayasu, \cite{Oka05}. Graphs of groups of finite type correspond to actions of virtually free groups on trees \cite[Theorem 7.3]{SW79}. Like the algebra considered in \cite{BMPST17}, our construction works for a much larger class of examples.
Pimsner \cite{Pim86} has used graphs of groups to decompose $KK$-groups for crossed products algebras, while Julg and Valette \cite{JV84} used similar ideas to investigate $K$-amenability of groups via their action on trees. More recently, graphs of $C^*$-algebras \cite{FF14} have been introduced to investigate the $K$-theory of quantum groups.

In \cite{BMPST17}, a generator and relations picture of an operator algebra associated to a graph of groups was presented, and we recall this in Section \ref{sec:prelims}. The associated algebra was shown to be Morita equivalent to a crossed product, and so has a groupoid model. The crossed product arises from the action of the fundamental group of the graph of groups on the boundary of its universal covering tree. Groupoid models are helpful because the associated algebra can be described in terms of a dense subalgebra of functions, and the KMS-states can be described using the machinery of \cite{Ren80}.

Groupoid models are not so helpful for computing $K$-theory, and in this paper we supplement the groupoid picture with a Cuntz-Pimsner description in Section~\ref{sec:cpmodel}. The correspondence from which the Cuntz-Pimsner algebra is constructed is simple enough to make $K$-theory computations tractable, as we show in Section~\ref{sec:Ktheory}. This correspondence, described in Section \ref{sec:gog-cozzie}, is built from direct summands related to both Rieffel imprimitivity bimodules \cite{Rie74Ind} and Kaliszewski-Larsen-Quigg's subgroup correspondences \cite{KLQ18}. The construction of the correspondence is also reminiscent of the correspondences associated to the (dual of) a directed graph.

In addition to $K$-theory computations,
identifying a Cuntz-Pimsner model for graph of groups algebras allows one to import considerable existing technology for their study. For example, the gauge invariant ideals can be classified \cite{Kat07}.  Cuntz-Pimsner algebras have a distinguished subalgebra, the fixed point of the gauge action called the core.
The core is often more amenable to study, being a direct limit, and ``supports'' many of the invariants of dynamical interest, \cite{Kat04cor,Pim97,CNNR11}.
A Cuntz-Pimsner model also provides a canonical choice of ``Toeplitz graph of groups algebra'', namely the Toeplitz algebra of the associated correspondence.

It is also worthy of note that having a Cuntz-Pimsner model presents a crossed product by a fairly complicated group as a ``generalised crossed product'' by the integers.

In Section~\ref{sec:ExelPardo}, by constructing a Morita equivalent correspondence, we obtain a Morita equivalent Cuntz-Pimsner algebra, which can be directly shown to be an Exel-Pardo algebra \cite{EP17}. In particular, from a graph of groups we construct a directed graph acted upon by a group with a relevant cocycle, so that the associated Exel-Pardo algebra agrees with our Morita equivalent Cuntz-Pimsner algebra.

In addition to $K$-theory, one can also consider $K$-homology as an invariant of a $C^*$-algebra. To access $K$-homological information one often uses index theory and the pairing with $K$-theory. Many of the graphs of groups $C^*$-algebras we consider have torsion in their $K$-theory and so this method does not apply \cite[Chapter 7]{HR00}. We show in Section~\ref{sec:poincare}, using the Poincar\'e duality techniques of \cite{RRS19}, that we can access the $K$-homology of an important class of examples, Baumslag-Solitar groups acting on the boundary of certain regular trees.

\subsection*{Acknowledgements}

Both authors would like to thank Aidan Sims for discussions about amenability, and the first author would like to thank Nathan Brownlowe and Jack Spielberg for their many insights into graph of groups $C^*$-algebras.

\section{Preliminaries}
\label{sec:prelims}

\subsection{Graphs of groups}
We begin by recalling the notion of an (undirected) graph.

\begin{dfn}
	A \hl{graph} $\Gamma = (\Gamma^0,\Gamma^1,r,s)$  consists of a countable sets of \hl{vertices} $\Gamma^0$ and \hl{edges} $\Gamma^1$ together with a \hl{range map} $r \colon \Gamma^1 \to \Gamma^0$ and \hl{source map} $s \colon \Gamma^1 \to \Gamma^0$. We also assume that there is an involutive ``edge reversal'' map $e \mapsto \ol{e}$ on $\Gamma^1$ satisfying $e \ne \ol{e}$ and $r(e) = s(\ol{e})$.
	An \hl{orientation} of a graph $\Gamma$ is a collection of edges $\Gamma_+^1 \subseteq \Gamma^1$ such that for each $e \in \Gamma^1$ precisely one of $e$ or $\ol{e}$ is contained in $\Gamma_+^1$.
\end{dfn}

We always assume that a graph comes with an orientation, but all constructions we consider are independent of this choice of orientation. With a choice of orientation,  a graph can be considered as a directed graph in which each edge has a `ghost' edge pointing in the opposite orientation.

Graphs are usually represented by a diagram. For a graph $\Gamma$ and $e \in \Gamma^1_+$ we either draw the edge $\ol{e}$ with a dashed line, as on the left of following figure, or omit it entirely as on the right:
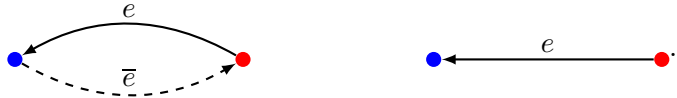
\[
	\begin{tikzpicture}
		[baseline=0.25ex,vertex/.style={circle, fill=black, inner sep=2pt}]

		\node[vertex,blue] (a) at (0,0) {};%
		\node[vertex,red] (b) at (3,0) {};%
		\node[inner sep = 2pt, anchor=south] at (a.north) {};%
		\node[inner sep = 2pt, anchor=south] at (b.north) {};%
		\draw[-latex,thick] (b) to [in=30, out=150] node[midway, anchor=south,inner sep=2.5pt] {$e$} (a);
		\draw[-latex,thick,dashed] (a) to [in=210, out=330] node[midway, anchor=south,inner sep=2.5pt] {$\ol{e}$} (b);
	\end{tikzpicture}
	\qquad \qquad \qquad
	\begin{tikzpicture}
		[baseline=0.25ex,vertex/.style={circle, fill=black, inner sep=2pt}]

		\node[vertex,blue] (a) at (0,0) {};%
		\node[vertex,red] (b) at (3,0) {};%
		\node[inner sep = 2pt, anchor=south] at (a.north) {};%
		\node[inner sep = 2pt, anchor=south] at (b.north) {};%
		\draw[-latex,thick] (b) -- (a)
		node[pos=0.5, anchor=south,inner sep=2.5pt] {$e$} ;
	\end{tikzpicture}.
\]

We say that a graph $\Gamma$ is \hl{locally finite} if $|r^{-1}(v)|<\infty$ for all $v \in \Gamma^0$, and \hl{nonsingular} $r^{-1}(r(e)) \ne \{e\}$ for all $e \in \Gamma^1$. In particular, a graph is nonsingular if and only if every vertex has valence strictly greater than $1$.
For each $n \in \NN$ we write $\Gamma^n = \{ e_1e_2\cdots e_n \mid e_i \in \Gamma^1, s(e_i) = r(e_{i+1}) \}$ for the collection of \hl{paths of length} $n$. Our convention for the direction of paths is the same as that of \cite{Rae05}. A \hl{circuit} is a path $e_1 \cdots e_k \in \Gamma^k$ for some $k \in \NN$ such that $r(e_1) = s(e_k)$ and $e_i \ne \ol{e_{i+1}}$ for all $1 \le i \le k-1$ and $e_{k-1} \ne \ol{e_1}$.
A \hl{tree} is a non-empty connected graph without circuits.

A group $H$ is said to act on the left of a tree $X = (X^0,X^1,r,s)$ if it acts on the left of both $X^0$ and $X^1$ in such a way that $r(h \cdot e) = h \cdot r(e)$ and $s(h \cdot e) = h \cdot s(e)$ for all $h \in H$ and $e \in X^1$. We say that $H$ acts \hl{without inversion} if $h \cdot e \ne \ol{e}$ for all $h \in H$ and $e \in X^1$. By performing barycentric subdivision on the edges of $X$, the group $H$ can always be made to act without inversion \cite[Section 3.1]{Ser80}. If $H$ acts without inversion on $X$, then there is a well-defined quotient graph $H\bs X = (H\bs X^0,H\bs X^1,r_{H \bs X},s_{H \bs X})$, with $r_{H \bs X} (He) = Hr(e) $, $s_{H \bs X} (He) = Hs(e)$, and $\ol{He} = H \ol{e}$ for all $e \in X^1$.

Bass-Serre Theory \cite{Ser80,Bas93} asks what additional data on the quotient graph $H \bs X$ is required to reconstruct the original action of $H$ on $X$.
The central idea is that this additional data can be captured in a ``graph of groups'' which we now define.

\begin{dfn}
	A {\em graph of groups} $\Gg = (\Gamma,G)$ consists of a connected graph $\Gamma$ together with:
	\begin{enumerate}
		\setlength\itemsep{0.5em}
		\item a discrete group $G_x$ for each vertex $x \in \Gamma^0$;
		\item a discrete group $G_e$ for each edge $e \in \Gamma^1$, with $G_e = G_{\ol{e}}$; and
		\item a monomorphism $\alpha_e\colon G_e \to G_{r(e)}$ for each $e \in \Gamma^1$.
	\end{enumerate}
\end{dfn}

We give a brief summary of the duality between graphs of groups and group actions on trees.  The interested reader is directed towards \cite{Ser80,Bas93,DD89}, or the introduction of \cite{BMPST17} for further details.

Suppose that $H$ acts on a tree $X=(X^0,X^1,r,s)$ without inversions. Set $\Gamma := H \bs X$. Fix a fundamental domain $Y = (Y^0,Y^1,r,s)$ of the action of $H$ on $X$. That is, $Y$ is a connected subgraph of $X$ such that $Y^1$ contains precisely one edge from each edge orbit. Note that $Y^0$ could contain more than one vertex from each vertex orbit.

For each $v \in Y^0$ we associate  to the vertex $Hv \in \Gamma^0$ the group $G_{Hv} :=\stab_H(v)$, and for each $e \in Y^1$ we associate  to $He \in \Gamma^1$ the group $G_{He} := \stab_H (e)$. Note that if $v,\,w \in Y^0$ are in the same orbit under the action of $H$, then $\stab_H(v) \cong \stab_H(w)$ so up to isomorphism our choice of vertex groups is consistent.
For each $e \in  Y^1$  the inclusion of $\stab_H(e)$ in $\stab_H(r(e))$ induces a monomorphism $\alpha_{He} \colon G_{He} \to G_{r(He)}$.
The graph of groups $\Gg = (\Gamma = H \bs X,G)$ constructed in this way is called the \hl{quotient graph of groups} for the action of $H$ on $X$.

Conversely, if one starts with a graph of groups $\Gg = (\Gamma,G)$, then there is a group $\pi_1(\Gg)$, called the \hl{fundamental group} of $\Gg$, and a \hl{universal covering tree} $X_\Gg$, both constructed from paths in $\Gg$. The fundamental group $\pi_1(\Gg)$ acts naturally, without inversions, on the tree $X_\Gg$. We will not make explicit use of the construction of the fundamental group nor the universal covering tree, and as such direct the interested reader to  \cite{Ser80,Bas93,BMPST17} for their construction.

\emph{The crux of Bass-Serre theory is that studying discrete group actions on trees is equivalent to studying graphs of groups. In particular, the quotient graph of groups construction and the construction of the fundamental group acting on the universal covering tree are mutually inverse \cite[Theorem 13]{Ser80}.}

We impose conditions on the graphs of groups that we work with.
\begin{dfn}
	A graph of groups $\Gg = (\Gamma,G)$ is said to be \hl{locally finite} if the underlying graph $\Gamma$ is locally finite and for all $e \in \Gamma^1$, the subgroup  $\alpha_e(G_e)$ has finite index in $G_{r(e)}$. A graph of groups $\Gg = (\Gamma,G)$ is said to be \hl{nonsingular} if for all $e \in \Gamma^1$ with $r^{-1}(r(e)) = \{e\}$ the monomorphism $\alpha_e\colon G_e \to G_{r(e)}$ is not an isomorphism.
\end{dfn}
Local finiteness and nonsingularity of $\Gg$ is equivalent to local finiteness and nonsingularity of the associated universal covering tree $X_\Gg$. \emph{All graphs of groups considered in this article will be locally finite and nonsingular.}

\begin{example}\label{ex:PSL2Z}
	The modular group $H:= PSL_2(\ZZ) \cong \ZZ_2 * \ZZ_3 = \langle a ,b \mid a^2 = b^3 = 1 \rangle$ acts without inversion on the $(2,3)$-regular tree,
	
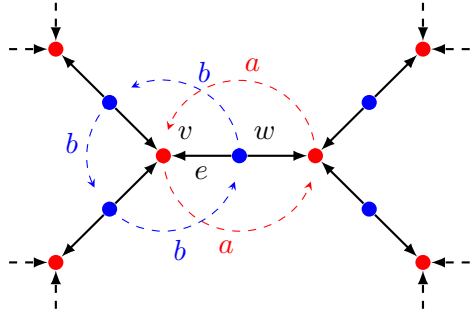
\[
		\begin{tikzpicture}
			[baseline=0.25ex,vertex/.style={circle, fill=black, inner sep=2pt}]

			\node[vertex,blue,label={[label distance=0.01]60:{$w$}}] (a) at (0,0) {};%

			\node[vertex,red,label={[label distance=0.02]60:{$v$}}] (b) at (-1,0) {};%
			\node[vertex,red] (ab) at (1,0) {};%

			\node[vertex,blue] (ba) at ($(b) + (135:1)$) {};%
			\node[vertex,blue] (bba) at ($(b) + (225:1)$) {};%
			\node[vertex,blue] (bab) at ($(ab) + (45:1)$) {};%
			\node[vertex,blue] (bbab) at ($(ab) + (-45:1)$) {};%

			\node[vertex,red] (aba) at ($(ba) + (135:1)$) {};%
			\node[vertex,red] (abba) at ($(bba) + (225:1)$) {};%
			\node[vertex,red] (abab) at ($(bab) + (45:1)$) {};%
			\node[vertex,red] (abbab) at ($(bbab) + (-45:1)$) {};%

			\node (baba) at ($(aba) + (90:0.75)$) {};%
			\node (bbaba) at ($(aba) + (180:0.75)$) {};%
			\node (babba) at ($(abba) + (180:0.75)$) {};%
			\node (bbabba) at ($(abba) + (270:0.75)$) {};%
			\node (babab) at ($(abab) + (90:0.75)$) {};%
			\node (bbabab) at ($(abab) + (0:0.75)$) {};%
			\node (babbab) at ($(abbab) + (0:0.75)$) {};%
			\node (bbabbab) at ($(abbab) + (270:0.75)$) {};%

			\draw[thick,-latex] (a) --node[midway,below] {$e$} (b) ;
			\draw[thick,-latex]	(a) -- (ab);
			\draw[thick,-latex]	(ba) -- (b);
			\draw[thick,-latex]	(bba) -- (b);
			\draw[thick,-latex]	(bab) -- (ab);
			\draw[thick,-latex]	(bbab) -- (ab);
			\draw[thick,-latex]	(ba) -- (aba);
			\draw[thick,-latex]	(bba) -- (abba);
			\draw[thick,-latex]	(bab) -- (abab);
			\draw[thick,-latex]	(bbab) -- (abbab);

			\draw[dashed,thick,-latex] (baba) -- (aba) ;
			\draw[dashed,thick,-latex] (bbaba) -- (aba);
			\draw[dashed,thick,-latex] (babba) -- (abba);
			\draw[dashed,thick,-latex] (bbabba) -- (abba);
			\draw[dashed,thick,-latex] (babab) -- (abab) ;
			\draw[dashed,thick,-latex] (bbabab) -- (abab);
			\draw[dashed,thick,-latex] (babbab) -- (abbab);
			\draw[dashed,thick,-latex] (bbabbab) -- (abbab);

			\draw[-stealth,blue,dashed] (a) arc(0:115:1) node[midway,above] {$b$} (ba);
			\draw[-stealth,blue,dashed] (ba) arc(135:205:1) node[midway,left] {$b$} (bba);
			\draw[-stealth,blue,dashed] (bba) arc(225:340:1) node[midway,below] {$b$} (a);

			\draw[-stealth,red,dashed] (ab) arc(0:160:1) node[midway,above] {$a$}(b);
			\draw[-stealth,red,dashed] (b) arc(180:340:1) node[midway,below] {$a$} (ab);
		\end{tikzpicture}.
	\]
	The generator $a$ acts by rotation around a $2$-valent vertex $w$, while $b$ acts by rotation around an adjacent $3$-valent vertex $v$. The orbits of vertices are coloured either red or blue in the above figure. As such, the quotient graph $\Gamma$ consists of $1$ edge and $2$ vertices. We also have $ \stab_H(v) = \langle b \mid b^3 = 1 \rangle$ and $\stab_H(w) = \langle a \mid a^2 = 1 \rangle$.
	The associated graph of groups the quotient graph of groups $\Gg = (\Gamma,G)$ is drawn in the following way:
	
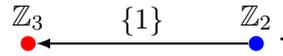
\[
		\begin{tikzpicture}
			[baseline=0.25ex,vertex/.style={circle, fill=black, inner sep=2pt}]

			\node[vertex,red] (a) at (0,0) {};%
			\node[vertex,blue] (b) at (3,0) {};%
			\node[inner sep = 2pt, anchor=south] at (a.north) {$\ZZ_3$};%
			\node[inner sep = 2pt, anchor=south] at (b.north) {$\ZZ_2$};%
			\draw[thick,-latex] (b) -- (a)
			node[pos=0.5, anchor=south,inner sep=2.5pt] {$\{1\}$} ;
		\end{tikzpicture}.
		\vspace*{10pt}
	\]
	Note that we have suppressed the inclusion maps from $\{1\}$ into both $\ZZ_3$ and $\ZZ_2$.
	Here $\Gg$ is both locally finite and nonsingular.
\end{example}

\begin{example}\label{ex:DinfBS}
	The Baumslag-Solitar group
	\[
		BS(1,2) = \langle a ,t \mid tat^{-1} = a^2 \rangle,
	\]
	acts on the infinite $3$-regular tree $X$ in a manner which we now spend some time describing.

	Fix a vertex $v \in X^0$ and an edge $e \in X^1$ with $r(e) = v$. Consider the infinite rooted binary tree $X'$ with edge set equal to all the edges of $X^1 \setminus \{e,\ol{e}\}$ which are in the same connected component as $v$. We label the vertices of $X'$ with words from the alphabet $\{0,1\}$ using the binary structure of the tree. In particular, the vertices adjacent to $v$ in $X'$ are labelled by $0$ and $1$, and inductively if a vertex is labelled by $w_1 \cdots w_k \in \{0,1\}^k$, then it is adjacent to the vertices labelled by $w_1 \cdots w_k 0$, $w_1 \cdots w_k 1$, and $w_1 \cdots w_{k-1}$. The labelling can be observed in the following diagram:
	\def\y{1.4}
	\def\x{2.5}
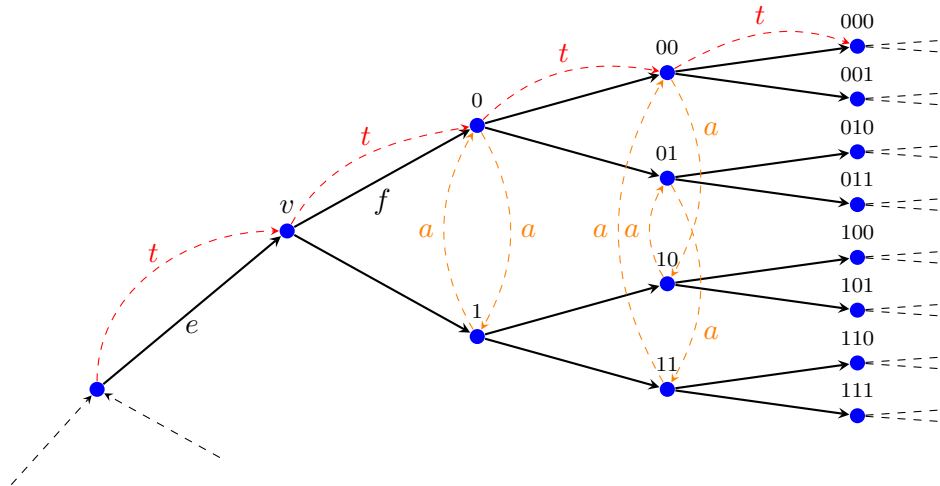
	\[
		\begin{tikzpicture}
			[baseline=0.25ex,vertex/.style={circle, fill=black, inner sep=2pt},label distance = 0.01]


			\node[vertex,blue] (w) at (-1*\x,-1.5*\y) {};%
			\node (w1) at ($(w) - 0.5*(\x,2*\y)$) {};%
			\node (w2) at ($(w) + 0.7*(\x,-1*\y)$) {};%

			\node[vertex,blue,label={[label distance=0.01]90:{$v$}}] (v) at (0,0) {};%

			\node[vertex,blue,label={[label distance=0.01]90:{\scriptsize$0$}}] (0) at (1*\x,\y) {};%
			\node[vertex,blue,label={[label distance=0.01]90:{\scriptsize$1$}}] (1) at (1*\x,-1*\y) {};%

			\node[vertex,blue,label={[label distance=0.01]90:{\scriptsize$00$}}] (00) at (2*\x,1.5*\y) {};%
			\node[vertex,blue,label={[label distance=0.01]90:{\scriptsize$01$}}] (01) at (2*\x,0.5*\y) {};%
			\node[vertex,blue,label={[label distance=0.01]90:{\scriptsize$10$}}] (10) at (2*\x,-0.5*\y) {};%
			\node[vertex,blue,label={[label distance=0.01]90:{\scriptsize$11$}}] (11) at (2*\x,-1.5*\y) {};

			\node[vertex,blue,label={[label distance=0.01]90:{\scriptsize$000$}}] (000) at (3*\x,1.75*\y) {};%
			\node[vertex,blue,label={[label distance=0.01]90:{\scriptsize$001$}}] (001) at (3*\x,1.25*\y) {};%
			\node[vertex,blue,label={[label distance=0.01]90:{\scriptsize$010$}}] (010) at (3*\x,0.75*\y) {};%
			\node[vertex,blue,label={[label distance=0.01]90:{\scriptsize$011$}}] (011) at (3*\x,0.25*\y) {};
			\node[vertex,blue,label={[label distance=0.01]90:{\scriptsize$100$}}] (100) at (3*\x,-0.25*\y) {};
			\node[vertex,blue,label={[label distance=0.01]90:{\scriptsize$101$}}] (101) at (3*\x,-0.75*\y) {};
			\node[vertex,blue,label={[label distance=0.01]90:{\scriptsize$110$}}] (110) at (3*\x,-1.25*\y) {};%
			\node[vertex,blue,label={[label distance=0.01]90:{\scriptsize$111$}}] (111) at (3*\x,-1.75*\y) {};%

			\node (0000) at ($(000)+(0.5*\x,0.0625*\y)$) {};%
			\node (0001) at ($(000)+(0.5*\x,-0.0625*\y)$) {};%
			\node (0010) at ($(001)+(0.5*\x,0.0625*\y)$) {};%
			\node (0011) at ($(001)+(0.5*\x,-0.0625*\y)$) {};%
			\node (0100) at ($(010)+(0.5*\x,0.0625*\y)$) {};%
			\node (0101) at ($(010)+(0.5*\x,-0.0625*\y)$) {};%
			\node (0110) at ($(011)+(0.5*\x,0.0625*\y)$) {};%
			\node (0111) at ($(011)+(0.5*\x,-0.0625*\y)$) {};%
			\node (1000) at ($(100)+(0.5*\x,0.0625*\y)$) {};%
			\node (1001) at ($(100)+(0.5*\x,-0.0625*\y)$) {};%
			\node (1010) at ($(101)+(0.5*\x,0.0625*\y)$) {};%
			\node (1011) at ($(101)+(0.5*\x,-0.0625*\y)$) {};%
			\node (1100) at ($(110)+(0.5*\x,0.0625*\y)$) {};%
			\node (1101) at ($(110)+(0.5*\x,-0.0625*\y)$) {};%
			\node (1110) at ($(111)+(0.5*\x,0.0625*\y)$) {};%
			\node (1111) at ($(111)+(0.5*\x,-0.0625*\y)$) {};%

			\draw[thick,-stealth]
			(v) -- (0) node[below,midway] {$f$};
			;%

			\draw[thick,-stealth]
			(v) edge (1)
			;%

			\draw[thick,-stealth]
			(0) edge (00)
			edge (01)
			;%

			\draw[thick,-stealth]
			(00) edge (000)
			edge (001)
			;%

			\draw[thick,-stealth]
			(01) edge (010)
			edge (011)
			;%

			\draw[thick,-stealth]
			(1) edge (10)
			edge (11)
			;%

			\draw[thick,-stealth]
			(10) edge (100)
			edge (101)
			;%

			\draw[thick,-stealth]
			(11) edge (110)
			edge (111)
			;%

			\draw[thick,-stealth]
			(w) -- (v) node[below,midway] {$e$}
			;%

			\draw[dashed,-stealth]
			(w1) edge (w)
			(w2) edge (w)
			;%

			\draw[dashed]
			(000) -- (0000)
			(000) -- (0001)
			(001) -- (0010)
			(001) -- (0011)
			(010) -- (0100)
			(010) -- (0101)
			(011) -- (0110)
			(011) -- (0111)
			(100) -- (1000)
			(100) -- (1001)
			(101) -- (1010)
			(101) -- (1011)
			(110) -- (1100)
			(110) -- (1101)
			(111) -- (1110)
			(111) -- (1111)
			;%

			\draw[-stealth,red,dashed] (w) to[out = 90, in = 180, looseness=1, edge node={node [midway,above] {$t$}}] (v);

			\draw[-stealth,red,dashed] (v) to[out = 60, in = 190, looseness=1, edge node={node [midway,above] {$t$}}] (0);

			\draw[-stealth,red,dashed] (0) to[out = 45, in = 170, looseness=1, edge node={node [midway,above] {$t$}}] (00);

			\draw[-stealth,red,dashed] (00) to[out = 30, in = 160, looseness=1, edge node={node [midway,above] {$t$}}] (000);

			\draw[-stealth,orange,dashed] (0) to[out = -60, in = 60, looseness=1, edge node={node [midway,right] {$a$}}] (1);

			\draw[-stealth,orange,dashed] (1) to[out = -240, in = 240, looseness=1, edge node={node [midway,left] {$a$}}] (0);

			\draw[-stealth,orange,dashed] (00) to[out = -60, in = 60, looseness=1, edge node={node [pos=0.25,right] {$a$}}] (10);

			\draw[-stealth,orange,dashed] (10) to[out = -240, in = 240, looseness=1, edge node={node [midway,left] {$a$}}] (01);

			\draw[-stealth,orange,dashed] (01) to[out = -60, in = 60, looseness=1, edge node={node [pos=0.75,right] {$a$}}] (11);

			\draw[-stealth,orange,dashed] (11) to[out = -240, in = 240, looseness=1, edge node={node [midway,left] {$a$}}] (00);

		\end{tikzpicture}.
	\]
	Fix a doubly infinite path $\mu$ in $X$ containing $e$ and passing through the vertices labelled by $0,\,00,\,000,\ldots$. The generator $t$ of $G$ acts hyperbolically on $X$ by translating $X$ along the infinite path $\mu$ in such a way that $t \cdot v = 0$ and $t \cdot w_1\cdots w_k = 0w_1\cdots w_k$ in the subtree $X'$.

	The generator $a$ of $BS(1,2)$ fixes each edge in $X^1 \setminus (X')^1$ and acts via a so-called ``odometer action'' on the remaining vertices. In particular, the action is defined recursively by
	\[
		a \cdot w_1 w_2 \cdots w_k
		= \begin{cases}
			1 w_2 \cdots w_k           & \text{if } w_1 = 0  \\
			0 (a \cdot w_2 \cdots w_k) & \text{if } w_1 = 1.
		\end{cases}
	\]
	It is relatively straightforward to check that the actions of $t$ and $a$ on $X$ respect the relations of $BS(1,2)$, and therefore give a well-defined action of $BS(1,2)$ on $X$.
	The action is transitive on both edges and vertices. Consequently, the quotient graph $BS(1,2) \bs X$ is a loop. One checks that with $v$ and $f$ as in the diagram
	\begin{align*}
		\stab_{BS(1,2)}(v) & = \langle a \rangle \cong \ZZ, \quad \text{ and}           \\
		\stab_{BS(1,2)}(f) & = \stab_{BS(1,2)}(0) = \langle tat^{-1} \rangle \cong \ZZ.
	\end{align*}
	Using additive notation for the integers, in the quotient graph of groups for the action of $BS(1,2)$ on $X$ we have $\alpha_{BS(1,2)e} = \id$ and $\alpha_{BS(1,2)\ol{e}} \colon k \mapsto 2k$ which we represent diagrammatically as,
	
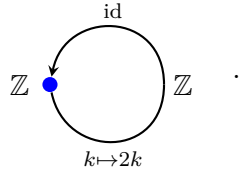
\[
		\begin{tikzpicture}[scale=2,baseline=0.5ex, vertex/.style={circle, fill=black, inner sep=2pt}]
			\begin{scope}
				\clip (-0.5,-.6) rectangle (1.2,.6);

				\node[vertex,blue] (0_0) at (0,0) [circle] {};


				\draw[-stealth,thick] (.75,0) .. controls (.75,.5) and (.1,.5) .. (0_0) node[pos=0, inner sep=3pt, anchor=west] {$\ZZ$}
				node[pos=0.5, anchor=south,inner sep=3pt]{$ \scriptstyle\id$};

				\draw[thick] (.75,0) .. controls (.75,-.5) and (.1,-.5) .. (0_0) node[pos=0.5, anchor=north,inner sep=3pt]{$\scriptstyle k \mapsto 2k$};

				\draw (-.2,0) node {$\ZZ$};

			\end{scope}

		\end{tikzpicture}.
	\]
	Similar considerations show that for all $m,n \in \ZZ$ the group $BS(m,n) = \langle a,t \mid t a^m t^{-1} = a^n \rangle $ acts on the $(|m|+|n|)$-regular tree, with the quotient graph of groups given by
	
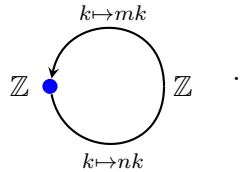
\[
		\begin{tikzpicture}[scale=2,baseline=0.5ex, vertex/.style={circle, fill=black, inner sep=2pt}]
			\begin{scope}
				\clip (-0.5,-.6) rectangle (1.2,.6);

				\node[vertex,blue] (0_0) at (0,0) [circle] {};


				\draw[-stealth,thick] (.75,0) .. controls (.75,.5) and (.1,.5) .. (0_0) node[pos=0, inner sep=3pt, anchor=west] {$\ZZ$}
				node[pos=0.5, anchor=south,inner sep=3pt]{$\scriptstyle k \mapsto mk$};

				\draw[thick] (.75,0) .. controls (.75,-.5) and (.1,-.5) .. (0_0) node[pos=0.5, anchor=north,inner sep=3pt]{$\scriptstyle k \mapsto nk$};

				\draw (-.2,0) node {$\ZZ$};

			\end{scope}

		\end{tikzpicture}.
	\]
\end{example}

\subsection{$C^*$-algebras associated to graphs of groups}
In \cite{BMPST17} the authors introduced a $C^*$-algebra associated to a graph of groups using the notion of  $\Gg$-families.  A $\Gg$-family is analogous to a Cuntz-Krieger-family, which are used to define directed graph $C^*$-algebras (see \cite{Rae05}).
\begin{dfn} \label{dfn:Gfamily}
	Let $\Gg = (\Gamma,G)$ be a locally finite nonsingular graph of countable groups.
	For each $e \in \Gamma^1$ choose a \hl{transversal} (a system of distinct representatives) $\Sigma_e$ for $G_{r(e)} / \alpha_e(G_e)$ and suppose that $1_{G_{r(e)}} \in \Sigma_e$. A {\em $\Gg$-family} in a $C^*$-algebra $B$ is a collection of partial isometries
	$\{S_e \mid e \in \Gamma^1\}$ in $B$
	together with representations $U_x \colon G_x \to B$, $g\mapsto U_{x,g}$  of $G_x$ by partial unitaries for each
	$x\in\Gamma^0$ satisfying the relations:
	\begin{enumerate}[label={ (G\arabic*)},align=left]
		\setlength\itemsep{0.7em}
		\item \label{itm:G1} $U_{x,1}U_{y,1}=0$ for each $x,y \in\Gamma^0$ with
		      $x \not= y $;
		\item \label{itm:G2}$U_{r(e),\alpha_e(g)}S_e=S_e U_{s(e),\alpha_{\overline{e}}(g)}$
		      for each $e\in\Gamma^1$ and $g\in G_e$;
		\item \label{itm:G3}
		      $U_{s(e),1}=S_e^*S_e+S_{\overline{e}}S_{\overline{e}}^*$ for each
		      $e\in\Gamma^1$;   and
		\item \label{itm:G4} $\begin{aligned}[t]
				      S_e^*S_e=\sum_{\substack{ r(f)=s(e),\, \mu \in \Sigma_f \\ hf\not=
						      1 \overline{e}}} U_{s(e),\mu}S_fS_f^*U_{s(e),\mu}^*
			      \end{aligned}$ for each
		      $e\in\Gamma^1$.
	\end{enumerate}
\end{dfn}

We note that the relation \ref{itm:G4} is independent of the choice of transversals (see \cite[Remark 3.2]{BMPST17}).
The relations \ref{itm:G3} and \ref{itm:G4} are similar to the Cuntz-Krieger relations of directed graph $C^*$-algebras \cite[p. 6]{Rae05}. In contrast to the Cuntz-Krieger relations, \ref{itm:G3} implies that $S_eS_{\ol{e}} = 0$ despite the fact that $r(e) = s(\ol{e})$. The relation \ref{itm:G2} can be interpreted as intertwining the representations of $G_e$ induced by $\alpha_e$ and $\alpha_{\ol{e}}$.

\begin{ntn}
	If $e \in \Gamma^1$ and $g \in G_{r(e)}$, then we will write
	$
		S_{g e} := U_{r(e),g} S_e.
	$
\end{ntn}

\emph{ For the remainder of this article we will always assume that for a graph of groups $\Gg = (\Gamma,G)$ we have chosen transversals $\Sigma_e$ for $G_{r(e)} / \alpha_e(G_e)$ such that $1_{G_{r(e)}} \in \Sigma_e$. }
We use $\Gg$-families to associate a $C^*$-algebra to a graph of groups.

\begin{dfn}
	Let $\Gg = (\Gamma,G)$ be a locally finite nonsingular graph of countable groups. The {\em graph of groups $C^*$-algebra} $C^*(\Gg)$ is the unique (up to isomorphism) $C^*$-algebra generated by a $\Gg$-family $\{u_{x},s_e \mid x \in \Gamma^0, e \in \Gamma^1\}$ which is universal in the following sense: if $\{U_x , S_e \mid x \in \Gamma^0, e \in \Gamma^1\}$ is a $\Gg$-family in a $C^*$-algebra $B$, then there is a unique $*$-homomorphism from $\Phi \colon C^*(\Gg) \to B$ such that $\Phi(u_x) = U_x$ and $\Phi(s_e) = S_e$.
\end{dfn}

This definition of a graph of groups $C^*$-algebra is analogous to the definition of the full group $C^*$-algebra. We will never consider a ``reduced'' version of the definition.
A concrete $\Gg$-family is constructed in \cite[Remark 3.4]{BMPST17}.
In general, the $C^*$-algebras of graphs of groups are not $C^*$-algebras of directed graphs (see \cite[Example 3.13]{BMPST17}). On the other hand, if $\Gg$ is a locally finite, nonsingular graph of groups with {\em trivial} edge groups, then according to \cite[Theorem 3.6]{BMPST17} there is a  directed graph $E_\Gg$ such that $C^*(E_\Gg) \cong C^*(\Gg)$.
As we shall see in Section~\ref{sec:Ktheory}, if the edge groups are non-trivial then the $K_1$-groups of $C^*(\Gg)$ often have torsion and are therefore not the $C^*$-algebras of directed graphs.

We recall the definition of the boundary $\partial X$ of a tree $X$ from \cite[Definition 2.18]{BMPST17}.
\begin{dfn}
	Let $X$ be a locally finite non-singular tree and fix $x \in X^0$. The \hl{boundary} of $X$ (at $x$) is the collection of infinite paths
	\[
		x\partial X = \{e_1e_2e_3\cdots \mid e_i \in X^1, r(e_1) = x,\, r(e_{i+1}) = s(e_i),\, \text{and } e_i \ne \ol{e_{i+1}} \text{ for all } i \in \NN \}
	\]
	equipped with the topology generated by the cylinder sets
	\[
		\Zz(e_1e_2\cdots e_k) = \{f_1f_2f_3 \cdots \in x \partial X \mid f_i = e_i \text{ for all } 1 \le i \le k \}
	\]
	ranging over all finite paths $e_1 e_2 \cdots e_k \in X^k$ such that $r(e_1) = x$ and $e_i \ne \ol{e_{i+1}}$.
\end{dfn}
The boundary $x \partial X$ is a totally disconnected compact Hausdorff space. Moreover, if $X$ is connected, then up to homeomorphism the boundary is independent of the choice of $x \in X^0$. In this case we simply write $\partial X$ for the boundary.
As noted in \cite[Remark 2.19]{BMPST17}, the above notion of boundary agrees with the Gromov boundary of $X$.

The action of a group on a tree extends naturally to an action by homeomorphisms on the boundary. In particular, the fundamental group $\pi_1(\Gg)$ of a graph of groups $\Gg$ acts naturally on the boundary $ \partial X_\Gg$ of its universal covering tree.

Finally, we mention the main theorem of \cite{BMPST17} which can be interpreted as a $C^*$-algebraic analogue of the Bass-Serre Theorem.
\begin{thm}[{\cite[Theorem 4.1]{BMPST17}}] \label{thm:bmpst17iso}
	Let $\Gg = (\Gamma,G)$ be a locally finite nonsingular graph of countable groups. Then  $C^*(\Gg)$ is isomorphic to $\Kk(\ell^2(\Gamma^0)) \otimes (C(\partial X_\Gg) \rtimes_\tau \pi_1(\Gg))$, where $\tau$ is the action on $C(\partial X_\Gg)$ induced by the action of $\pi_1(\Gg)$ on the universal covering tree $X_\Gg$.
\end{thm}

\subsection{Group action cocycles} \label{sec:cocycles}
To streamline computations in $C^*(\Gg)$, we introduce group action cocycles. They allow us to ``pass group elements through edges'' in the sense of Lemma~\ref{lem:carrying} below.

\begin{dfn}
	Let $G$ and $H$ be groups and suppose that $G$ acts on a set $X$. An $H$-valued $1$-cocycle for the action of $G$ on $X$ is a map $c \colon G \times X \to H$ such that
	\begin{equation}\label{eq:cocycle}
		c(g_1g_2,x) = c(g_1,g_2\cdot x) c(g_2,x)
	\end{equation}
	for all $g_1,\,g_2\in G$ and $x \in X$.
\end{dfn}
Let $\Gg = (\Gamma,G)$ be a graph of groups and recall that $\Sigma_e \subseteq G_{r(e)}$ is a transversal for the coset space $G_{r(e)} / \alpha_e(G_e)$. For each $e \in \Gamma^1$ the group $G_{r(e)}$ acts canonically by left translation on $G_{r(e)} / \alpha_e(G_e)$, and this induces an action of $G_{r(e)}$ on $\Sigma_e$. There is a $G_e$-valued cocycle $c_{e} \colon G_{r(e)} \times \Sigma_e \to G_e$ for the action of $G_{r(e)}$ on $\Sigma_e$ given by
\begin{equation}\label{eq:cocyclee}
	c_{e}(g,\mu) := \alpha_e^{-1}((g \cdot \mu)^{-1} g \mu)
\end{equation}
for all $g \in G_{r(e)}$ and $\mu \in \Sigma_e$. In the product $(g \cdot \mu)^{-1} g \mu$ we are considering both $\mu$ and $g \cdot \mu$ as elements of $G_{r(e)}$. Then for each $g \in G_{r(e)}$ and $\mu \in \Sigma_e$ we have
\begin{equation}\label{eq:splitting}
	g\mu = (g \cdot \mu) \alpha_e(c_e(g,\mu)).
\end{equation}
In particular, $g = (g \cdot 1) \alpha_e(c_e(g,1))$.

\begin{lem}\label{lem:carrying}
	Let $\Gg = (\Gamma,G)$ be a locally finite nonsingular graph of countable groups and suppose that $\{U_x \mid x \in \Gamma^0\} \cup \{S_e \mid e \in \Gamma^1\}$ is a $\Gg$-family. Then for all $e \in \Gamma^1$, $g \in G_{r(e)}$, and $\mu \in \Sigma_e$ we have,
	\begin{equation}\label{eq:carrying}
		U_{r(e),g} S_{\mu e} = S_{(g \cdot \mu)e} U_{s(e), c_{e}(g,\mu)}.
	\end{equation}

\end{lem}
\begin{proof}
	This follows immediately from \eqref{eq:splitting} and \ref{itm:G2}.
\end{proof}

Now suppose that $f \in \Gamma^1$ is such that $s(f) = r(e)$ and $f \ne \ol{e}$. Then the cocycle $c_e$ restricts to a cocycle $c_{fe} \colon G_f \times \Sigma_e \to G_e$ for the induced action of $G_f$ on $\Sigma_e$ via $\alpha_{\ol{f}}$. Explicitly, for $g \in G_f$ and $\mu \in \Sigma_e$ we define
\begin{equation}
	c_{fe} (g,\mu) :=c_e(\alpha_{\ol{f}}(g),\mu).
	\label{eq:cee-eff-ee}
\end{equation}
In the case where $f = \ol{e}$, the set $\Sigma_e \setminus \{1\}$ is invariant under the left action of $G_{\ol{e}} = G_e$ on $\Sigma_e$. In particular, $c_e$ restricts to a cocycle $c_{\ol{e}e} \colon G_{\ol{e}} \times (\Sigma_e \setminus \{1\}) \to G_e$ given by
\[
	c_{\ol{e}e} (g, \mu) = c_e(\alpha_{e}(g),\mu).
\]

\subsection{$C^*$-correspondences and Cuntz-Pimsner algebras}

We  refer to \cite{Lan95} or \cite{RW98} for background on $C^*$-correspondences and their $C^*$-algebras.
Given a right Hilbert $A$-module $E$, we write $\End_A(E)$ for the $C^*$-algebra of adjointable operators on $E$. For $\xi,\,\eta \in E$ we write $\Theta_{\xi,\eta}$ for the rank-one operator defined by $\Theta_{\xi,\eta}(\zeta) = \xi \cdot (\eta \mid \zeta)_A$ for all $\zeta \in E$. The closed 2-sided ideal of generalised compact operators is denoted
\[
	\End_{A}^0(E) := \ol{\spaan} \{\Theta_{\xi,\eta} \mid \xi,\eta \in E\} \triangleleft \End_A(E).
\]
A right Hilbert $A$-module $E$ is said to be \hl{full} if $\ol{(E \mid E)_A} = A$.

\begin{dfn}
	An \hl{$A$--$B$-correspondence} is a pair $(\phi,E)$ consisting of a right Hilbert $B$-module $E$ together with a $*$-homomorphism $\phi \colon A \to \End_B(E)$, which defines a left action of $A$ on $E$. We say that $(\phi,E)$ is \hl{faithful} if $\phi$ is injective and \hl{nondegenerate} if $\ol{\phi(A)E} = E$.

	An $A$--$B$-correspondence $(\phi,E)$ is said to be an \hl{imprimitivity bimodule} if $E$ is full and comes with a left $A$-linear inner product ${}_{A}(\,\cdot \mid \cdot\,)$ making it a full left Hilbert $A$-module which satisfies the imprimitivity relation \[\xi \cdot (\eta \mid \zeta)_B = \phi({}_A(\xi \mid \eta))\zeta\] for all $\xi,\,\eta,\, \zeta \in E$.
\end{dfn}

If $E$ is a right Hilbert $A$-module then we denote by $E^* = \{\xi ^* \mid \xi \in E\}$ the conjugate vector space of $E$. Then $E^*$ becomes a left Hilbert $A$-module with left action $a \cdot \xi^* = (\xi \cdot a^*)^*$ and inner product ${}_A( \xi^* \mid \eta^* ) = (\xi \mid \eta)_A$. If $(\phi,E)$ is an $A$--$B$-imprimitivity bimodule then $E^*$ becomes a $B$--$A$-imprimitivity bimodule such that $E \otimes_B E^* \cong A$ and $E^* \otimes_A E \cong B$~\cite[Proposition~3.28]{RW98}.

\begin{dfn}
	A \hl{representation} of an $A$--$A$-correspondence $(\phi,E)$  in a $C^*$-algebra $B$, denoted $(\pi,\psi) \colon (\phi,E) \to B$, consists of a $*$-homomorphism $\pi \colon A \to B$ together with a linear map $\psi \colon E \to B$ such that:
	\begin{enumerate}[label={ (\roman*)},align=right]
		\setlength\itemsep{0.5em}
		\item $\psi(\xi)\pi(a) = \psi(\xi \cdot a)$ for all $\xi \in E$ and $a \in A$,
		\item $\pi(a) \psi(\xi) = \psi( \phi(a) \xi)$ for all $a \in A$ and $\xi \in E$, and
		\item $\psi(\xi)^*\psi(\eta) = \pi((\xi \mid \eta)_A)$ for all $\xi,\,\eta \in E$.
	\end{enumerate}
	We denote by $C^*(\pi,\psi)$ the $C^*$-algebra generated by $\pi(A) \cup \psi(E)$ in $B$ and call this the \hl{$C^*$-algebra generated by $(\pi,\psi)$.}
\end{dfn}

Pimsner \cite{Pim97} introduced two universal $C^*$-algebras for representations of correspondences. The first is the Toeplitz algebra, introduced next. The second is the Cuntz-Pimsner algebra, corresponding to a restricted class of representations. Subsequent work by Muhly-Solel \cite{MS00} and Katsura \cite{Kat04cor} expanded the class of correspondences to which these algebras can be associated.

\begin{dfn}
	The \hl{Toeplitz algebra} $\Tt_E$ of an $A$--$A$-correspondence $(\phi,E)$ is the unique (up to isomorphism) $C^*$-algebra generated by a representation $(j_A,j_E)$ of $(\phi,E)$ satisfying the following universal property: if $(\pi,\psi) \colon (\phi,E) \to B$ is a representation then there is a unique $*$-homomorphism $\pi \times \psi \colon \Tt_E\to C^*(\pi,\psi)$ such that $(\pi \times \psi) \circ j_A = \pi$ and $(\pi\times \psi ) \circ j_E = \psi$.
\end{dfn}

Now suppose that $(\phi,E)$ is a full, faithful, and nondegenerate $A$--$A$-correspondence, and suppose that $\phi(A) \subseteq \End_A^0(E)$.
Given a representation $(\pi,\psi)\colon (\phi,E) \to B$ there is an induced $*$-homomorphism $\psi^{(1)} \colon \End_B^0(E) \to B$ satisfying $\psi^{(1)}(\Theta_{\xi,\eta}) = \psi(\xi)\psi(\eta)^*$ for all $\xi,\, \eta \in E$. We say that a representation $(\pi, \psi)$ of $(\phi,E)$ is \hl{Cuntz-Pimsner covariant} if $(\psi^{(1)} \circ \phi)(a) = \pi(a)$ for all $a \in A$.

\begin{dfn}
	The \hl{Cuntz-Pimsner algebra} $\Oo_E$ of an  $A$--$A$-correspondence $(\phi,E)$ is the unique (up to isomorphism) $C^*$-algebra generated by a Cuntz-Pimsner covariant representation $(i_A,i_E)$ of $E$ satisfying the following universal property: if $(\pi,\psi) \colon (\phi,E) \to B$ is a Cuntz-Pimsner covariant representation then there is a unique $*$-homomorphism $\pi \times \psi \colon \Oo_E\to C^*(\pi,\psi)$ such that $(\pi \times \psi) \circ i_A = \pi$ and $(\pi \times \psi) \circ i_E = \psi$.
\end{dfn}

We show in Theorem~\ref{thm:isomorphism} that the $C^*$-algebra of a graph of groups $C^*(\Gg)$ can be realised as a Cuntz-Pimsner algebra. In the next section we describe the correspondence that the aforementioned Cuntz-Pimsner algebra is built from.

\section{The graph of groups correspondence}
\label{sec:gog-cozzie}

One standard method of building a Cuntz-Pimsner model for directed graph algebras, \cite[Proposition 3.8]{RRS19}, starts with functions on vertices as the coefficient algebra, and builds a $C^*$-correspondence from functions on edges.
The coefficient algebra for our Cuntz-Pimsner model of $C^*(\Gg)$ is instead the direct sum of matrices over the group $C^*$-algebras of the {\em edge} groups.
In the case that our edge groups are all trivial, our correspondence is akin to the directed graph correspondence for the dual graph, \cite[Corollary 2.6]{Rae05}.

\subsection{The graph of groups correspondence}
We begin by considering a path of length two, $fe \in \Gamma^2$. We associate to the path $fe$ a $C^*(G_f)$--$C^*(G_e)$-correspondence $D_{fe}$. The correspondence $D_{fe}$ is analogous to the discrete case of the subgroup correspondences considered in \cite{KLQ18}; the difference being that we do not assume that $G_f = G_e$.

\begin{dfn}[The algebras]
	For each $e \in \Gamma^1$ let $B_e := C^*(G_e)$, the full group $C^*$-algebra of $G_e$.
	Let $A_e:= M_{\Sigma_e}(B_e)$ denote the $C^*$-algebra of $|\Sigma_e| \times |\Sigma_e|$ matrices over $B_e$.
\end{dfn}

We define a right $B_e$-module $F_e$ as follows.
Consider the $*$-algebra of compactly supported functions $C_c(G_e)$ on $G_e$, equipped with the convolution product. Define a right action of  $ a\in C_c(G_e)$ on $ \xi \in C_c(G_{r(e)})$ by
\begin{equation*}
	(\xi \cdot a)(h) = \sum_{g \in G_{e}} \xi(h \alpha_e(g))\, a(g^{-1}),
\end{equation*}
and define a $C_c(G_e)$-valued inner product on $\xi,\, \eta \in C_c(G_{r(e)})$ by
\begin{equation} \label{eq:preip}
	(\xi \mid \eta)_{C_c(G_e)} (h) := \sum_{g \in G_{r(e)}} \ol{\xi(g)}\,\eta(g \alpha_e(h)) = \sum_{\mu \in \Sigma_e} \sum_{k \in G_e} \ol{\xi(\mu \alpha_e(k))}\,\eta(\mu \alpha_e(kh)).
\end{equation}
A norm on $C_c(G_e)$ is given by $\|\xi\| = \|(\xi \mid \xi)_{C_c(G_e)}\|_{full}^{1/2}$, where $\|\cdot\|_{full}$ denotes the full $C^*$-norm on $C_c(G_e)$. Then $C_c(G_{r(e)})$ can be completed into a right $B_e$-module (see \cite[p. 4]{Lan95} or \cite[Lemma 2.16]{RW98} for details).

\begin{dfn}[An imprimitivity bimodule]
	Let $F_e$ denote the right $B_e$-module given by completing $C_c(G_{r(e)})$ in the norm defined by the inner product \eqref{eq:preip}.
\end{dfn}
Recall that for each $e \in \Gamma^1$ the vertex group $G_{r(e)}$ acts canonically on the set of transversals $\Sigma_e$ for $G_{r(e)}/\alpha_e(G_e)$. As such, we can consider the associated (full) crossed-product $C^*$-algebra $C(\Sigma_e) \rtimes G_{r(e)}$.
We remark that the module $F_e$ is the $C(\Sigma_e) \rtimes G_{r(e)}$--$C^*(G_e)$-imprimitivy bimodule of Green's Imprimitivity Theorem \cite[Theorem C.23]{RW98}. In particular, $F_e$ comes equipped with a left action $\psi \colon C(\Sigma_e) \rtimes G_{r(e)} \to \End_{B_e}(F_e)$ which we spend a moment to describe.

Since $\Gg$ is locally finite, $\Sigma_e$ is a finite set for each $e \in \Gamma^1$. So we may view $C(\Sigma_e) \rtimes G_{r(e)}$ as a completion of $C_c(G_{r(e)} \times \Sigma_e)$, where $C_c(G_{r(e)} \times \Sigma_e)$ is a $*$-algebra under the operations
\[
	(a * b)(g,\mu) = \sum_{h \in G_{r(e)}} a(h,\mu) b(h^{-1}g, h^{-1} \cdot \mu)
	\quad \text{and} \quad
	a^*(g,\mu) = \ol{a(g^{-1},g^{-1} \cdot \mu)}
\]
for $a,\,b \in C_c(G_{r(e)} \times \Sigma_e)$.
The left action $\psi \colon C(\Sigma_e) \rtimes G_{r(e)} \to \End_{B_e}(F_e)$ then satisfies
\begin{equation}\label{eq:leftcrossedprod}
	\psi(a)\xi(h) = \sum_{k \in G_{r(e)}} a(k^{-1},h \cdot 1) \xi(kh)
\end{equation}
for all  $a \in C_c(G_{r(e)} \times \Sigma_e)$ and $\xi \in C_c(G_{r(e)})$,
where $h \cdot 1 \in \Sigma_e$ represents the coset $h \alpha_e(G_e)$. The module $F_e$ is also full and comes equipped with a left $C(\Sigma_e) \rtimes G_{r(e)}$-valued inner product so that $F_e^*$ is a $B_e$--$C(\Sigma_e) \rtimes G_{r(e)}$-correspondence. Although it may not be immediately apparent, $F_e$ is a free
right $B_e$-module.

\begin{lem}\label{lem:FisColumns}
	Let $e \in \Gamma^1$ and consider the right $B_e$-module $\bigoplus_{\mu \in \Sigma_e} B_e$.
	Denote the left action of $A_e \cong \End_A^0(\bigoplus_{\mu \in \Sigma_e} B_e)$ on $\bigoplus_{\mu \in \Sigma_e} B_e$ by $\ell$. Then $(\ell,\bigoplus_{\mu \in \Sigma_e} B_e)$ and $(\psi,F_e)$ are isomorphic $C^*$-correspondences.
\end{lem}
\begin{proof}
	Identify $C_c( \Sigma_e \times G_e)$ with $\bigoplus_{\mu \in \Sigma_E} C_c(G_e)$.
	A straightforward computation shows that the linear map $\Phi \colon C_c(G_{r(e)}) \to C_c(\Sigma_e \times G_e)$ defined by
	\begin{equation}\label{eq:twopicturesofF}
		\Phi(\xi)(\mu,k) = \xi(\mu \alpha_e(k)),\quad \mu\in\Sigma_e,\ k\in G_e,
	\end{equation}
	preserves both inner products and right actions, and therefore extends to an isometric linear map $\Phi\colon F_e \to \bigoplus_{\mu \in \Sigma_e} B_e$ of right $B_e$-modules. Since $\Phi$ sends the point mass at $\mu\alpha_e(k)$ to the point mass at $k$ in the $\mu$-th copy of $B_e$, $\Phi$ is surjective.

	Since $F_e$ is finitely generated, Green's Imprimitivity Theorem \cite[Theorem C.23]{RW98} implies that $\End_{B_e}(F_e) \cong C(\Sigma_e) \rtimes G_{r(e)}$. Consequently,  $C(\Sigma_e) \rtimes G_{r(e)} \cong A_e$, the left actions are preserved, and the $C^*$-correspondences $(\psi,F_e)$ and $(\ell,\bigoplus_{\mu \in \Sigma_e} B_e)$ are isomorphic.
\end{proof}
In the sequel, we frequently identify the correspondences $(\psi,F_e)$ and $(\ell,\bigoplus_{\mu \in \Sigma_e} B_e)$ in the manner described by Lemma~\ref{lem:FisColumns}. Using \eqref{eq:twopicturesofF} to identify $C_c(\Sigma_e \times G_e)$ as a dense subspace of $F_e$, the left action $\psi \colon C(\Sigma_e) \rtimes G_{r(e)} \to \End_{B_e}(F_e)$ of \eqref{eq:leftcrossedprod} satisfies
\begin{equation}\label{eq:intermediateleftaction}
	\psi(a)\xi(\mu,h) = \sum_{k \in G_{r(e)}} a(k^{-1},h \cdot 1)  \xi(k \cdot \mu, c_e(k,\mu) h)\quad \mu\in\Sigma_e,\ h\in G_e,
\end{equation}
for all $a \in C_c(G_{r(e)} \times \Sigma_e)$ and $\xi \in C_c(\Sigma_e \times G_e)$.

Our next goal is to construct a $B_f$--$B_e$-correspondence $D_{fe}$ from $F_e$ by restricting the left action of $C(\Sigma_e) \rtimes G_e$ to an action of $B_f$. In the special case where $f = \ol{e}$, the construction varies slightly, and so we introduce the following notation.
\begin{ntn}
	For each $fe \in \Gamma^2$ let $\Delta_{fe} \subseteq \Sigma_e$ be given by
	\[
		\Delta_{fe} := \begin{cases}
			\varnothing      & \text{ if } f \ne \ol{e} \\
			\{1_{G_{r(e)}}\} & \text{ if } f = \ol{e}.
		\end{cases}
	\]
\end{ntn}

\begin{dfn}
	For each $fe \in \Gamma^2$ define the right Hilbert $B_e$-module
	\[
		D_{fe} :=  \bigoplus_{
			\mu \in \Sigma_e \setminus \Delta_{fe}} B_{e}
	\]
	considered as a direct sum of right Hilbert $B_e$-modules. Note that if $f = \ol{e}$ and $\alpha_e$ is surjective (so $\Sigma_e = \{1_{G_{r(e)}}\}$), then $D_{\ol{e}e} = \{0\}$.
\end{dfn}

\begin{prop} \label{prop:Dfe}
	Let $fe \in \Gamma^2$, and suppose that either $f \ne \ol{e}$, or $\alpha_e$ is not surjective. Let $c_{fe} \colon G_f \times \Sigma_e \to G_e$ be the cocycle for the action
	of $G_f$ on $\Sigma_e$ given by \eqref{eq:cee-eff-ee}.  Then $(\ol{\varphi}_{fe},D_{fe})$ is a $B_f$--$B_e$-correspondence with left action $\ol{\varphi}_{fe} \colon B_f \to \End_{B_e}^0(D_{fe})$ satisfying
	\begin{equation}\label{eq:leftactionD}
		\ol{\varphi}_{fe}(a) \xi(\mu,h) = \sum_{k \in G_f} a(k^{-1}) \xi(\alpha_{\ol{f}}(k) \cdot \mu,c_{fe}(k,\mu)h)
	\end{equation}
	for all $a \in C_c(G_f)$ and $\xi \in C_c((\Sigma_e \setminus \Delta_{fe}) \times G_e)$. Moreover, $\ol{\varphi}_{fe}$ is unital and faithful.
\end{prop}

\begin{proof}
	Since $\Sigma_e$ is finite the universal property of $C^*(G_{r(e)})$ induces an injective $*$-homomorphism $\pi \colon C^*(G_{r(e)}) \to C(\Sigma_e) \rtimes G_{r(e)}$ such that $\pi(a)(g,\mu) = a(g)$ for all $a \in C_c(G_{r(e)})$.
	As $G_{r(e)}$ is discrete \cite[Proposition 1.2]{Rie74Ind} implies that the inclusion $C_c(G_f) \subseteq C_c(G_{r(e)})$---induced by $\alpha_{\ol{f}}\colon G_f \to G_{r(e)}$\,---extends to an injective $*$-homomorphism $\iota \colon B_f \to C^*(G_{r(e)})$.
	Hence, the composition $\psi \circ \pi \circ \iota$ defines a left action of $C^*(G_f)$ on $F_e$ by compact operators. Moreover, this action is unital and faithful. It follows from \eqref{eq:intermediateleftaction} and the definitions of $\pi$ and $\iota$ that $\psi \circ \pi \circ \iota = \ol{\varphi_{fe}}$ when $f \ne \ol{e}$.
	If $f = \ol{e}$, then the discussion in Section~\ref{sec:cocycles} implies that the action of $G_e$ on $\Sigma_e$ restricts to an action on $\Sigma_e \setminus \{1\}$. In particular, the action $\psi \circ \pi \circ \iota$ on $F_e$ restricts to an action on $\bigoplus_{\mu \in \Sigma_e \setminus \{1\}} B_e$ satisfying \eqref{eq:leftactionD}.
\end{proof}

When $f = \ol{e}$ the copy of $C^*(G_e)$ corresponding to the identity coset is removed from $F_e$ in order take care of the orthogonality of $S_e^*S_e$ and $S_{\ol{e}}S_{\ol{e}}^*$ introduced by condition \ref{itm:G3}. This will become more apparent later. Conceptually, we have passed to the dual graph in order to accommodate this orthogonality.

We assemble the $C^*$-correspondences $(\ol{\varphi}_{fe},D_{fe})$ into a single $C^*$-correspondence.
\begin{dfn}[The graph of groups module]
	Let
	\[
		B := \bigoplus_{e \in \Gamma^1} B_e \quad \text{and} \quad D := \bigoplus_{fe \in \Gamma^2} D_{fe}.
	\]
	We remind the reader that for each $e \in \Gamma^1$ we also have $\ol{e} \in \Gamma^1$.
	Here we think of $D_{fe}$ as a right Hilbert $B$-module, with $B_l$ acting trivially on $D_{fe}$ unless $l = e$. So $D$ is a direct sum of right Hilbert $B$-modules. We call $D$ the \hl{graph of groups module associated to $\Gg$}.
\end{dfn}

\begin{prop}\label{prop:D}
	Let $\Gg = (\Gamma,G)$ be a locally finite nonsingular graph of countable groups. Extend $\ol{\varphi}_{fe}$ to a left action $\ol{\varphi}_{fe} \colon B \to \End_{B}(D_{fe})$, where $\ol{\varphi}_{fe}(B_l) = \{0\}$ unless $l=f$.
	Then $D$ is full and $(\ol{\varphi}:=\oplus_{fe \in \Gamma^2}\ol{\varphi}_{fe}, D)$ is a $B$--$B$-correspondence with a faithful, non-degenerate left action satisfying $\ol{\varphi}(B) \subseteq \End_B^0(D)$.
\end{prop}

\begin{proof}
	Nonsingularity of $\Gg$ guarantees that for each edge $e \in \Gamma^1$ there exists some $f \in \Gamma^1$ with $s(f) = r(e)$ and $D_{fe} \ne \{0\}$. Since each summand $D_{fe}$ of $D$ is full as a right Hilbert $B_e$-module by Lemma~\ref{lem:Eidentify}, it follows that $D$ is full as a right Hilbert $B$-module. Nonsingularity of $\Gg$ also implies that for each $f \in \Gamma^1$ there always exists $e \in \Gamma^1$ such that either $f \ne \ol{e}$ or such that $\alpha_e$ is not surjective.
	It now follows from Proposition~\ref{prop:Dfe}  that $\ol{\varphi}_{fe}$ is injective for each ${fe}\in \Gamma^2$, and so $\ol{\varphi}$ is also injective.
	Since each $\ol{\varphi}_{fe}\colon B_f \to \End_{B_e} (D_{fe})$ is unital  $\ol{\varphi}$ is non-degenerate.

	To see that $\ol{\varphi}(B) \subseteq \End_B^0(D)$, recall from Proposition~\ref{prop:Dfe} that $\ol{\varphi}_{fe}(B_f) \subseteq \End_{B_e}^0(D_{fe})$. Since $\Gg$ is locally finite, for each $f \in \Gamma^1$ and $a \in B_f$ we have $\ol{\varphi}(a) = \sum_{r(e)=s(f)} \ol{\varphi}_{fe}(a)$, which belongs to $\bigoplus_{r(e)=s(f)} \End_{B}^0(D_{fe}) \subseteq \End_{B}^0(D)$. In general, each $a \in B$ can be approximated in norm by a finite sum of $a_f \in B_f$, so it follows that $\ol{\varphi}(a) \in \End_B^0(D)$.
\end{proof}

\begin{dfn}[The graph of groups correspondence]
	We call $(\ol{\varphi}, D)$ of Proposition~\ref{prop:D} the \hl{graph of groups correspondence associated to} $\Gg$.
\end{dfn}

\subsection{The amplified graph of groups correspondence}

The graph of groups correspondence $D$ proves to be very useful for the $K$-theory calculations in Section~\ref{sec:Ktheory} and will be analysed further in Section~\ref{sec:ExelPardo}, but to facilitate the construction of a Cuntz-Pimsner model
of  $C^*(\Gg)$, we pass to a Morita equivalent correspondence, which we call the \hl{amplified graph of groups correspondence}. By \cite{MS00}, the associated Cuntz-Pimsner algebras will also be Morita equivalent.
To begin we amplify $D_{fe}$ to an $A_f$--$A_e$-correspondence.
Recall that $F_e^*$ is a $B_e$--$A_e$-correspondence.
\begin{dfn}[Amplified correspondences]
	\label{defn:A-ee--E-ee}
	For each $fe \in \Gamma^2$ define a right $A_e$-module $E_{fe}$ as the balanced tensor product,
	\[
		E_{fe} := F_f \otimes_{B_f} D_{fe} \otimes_{B_e} F_e^*.
	\]
	Let $\varphi_{fe} \colon A_f \to \End_{A_e}^0(E_{fe})$ denote the left action induced by the left action of $A_f \cong \End_{B_f}^0(F_f)$ on $F_f$. Then $(\varphi_{fe},E_{fe})$ is an $A_f$--$A_e$-correspondence.
\end{dfn}
The graph of groups correspondence is now built in a similar manner to $(\ol{\varphi},D)$.

\begin{dfn}[The amplified graph of groups correspondence]
	\label{defn:AandE}
	Define
	\[
		A := \bigoplus_{e \in \Gamma^1} A_e = \bigoplus_{e \in \Gamma^1} M_{\Sigma_e} (C^*(G_e))
		\quad \text{and} \quad
		E := \bigoplus_{fe \in \Gamma^2} E_{fe}.
	\]
	Then $E$ is a right Hilbert $A$-module with right action analogous to that of $B$ on $D$. A left action $\varphi \colon A \to \End_{A}(E)$ is given by $\varphi = \oplus_{fe \in \Gamma^2} \varphi_{fe}$. We call the $A$--$A$-correspondence $(\varphi,E)$ the \hl{amplified graph of groups correspondence associated to} $\Gg$.
\end{dfn}

\begin{prop}
	Let $\Gg  = (\Gamma,G)$ be a locally finite nonsingular graph of countable groups. Then $E$ is full and the left action $\varphi\colon A \to \End_A(E)$ is faithful, non-degenerate, and by compact operators.
\end{prop}
\begin{proof}
	This follows from Proposition~\ref{prop:D} and the fact that each $F_e$ is a finitely generated $A_e$--$B_e$-imprimitivity bimodule.
\end{proof}

Passing to the amplified graph of groups correspondence induces a Morita equivalence on the associated Cuntz-Pimsner algebras.
\begin{prop}\label{prop:DandEmoritaeq}
	Let $\Gg  = (\Gamma,G)$ be a locally finite nonsingular graph of countable groups. Then $\Oo_E$ is Morita equivalent to $\Oo_D$.
\end{prop}
\begin{proof}
	Let $F = \bigoplus_{e \in \Gamma^1} F_e$.
	Using the (completed) direct sum of $C^*$-modules, we have
	\[
		E
		\cong \Big(\bigoplus_{f \in \Gamma^1} F_f\Big) \ox_{B} \Big(\bigoplus_{fe \in \Gamma^2} D_{fe}\Big) \ox_{B} \Big(\bigoplus_{e \in \Gamma^1} F_e^*\Big)
		=F \ox_B D \ox_B F^*.
	\]
	The Morita equivalence of $D$ and $E$ then passes to a Morita equivalence of $\Oo_D$ and $\Oo_{E}$, by \cite{MS00}.
\end{proof}

We now give a second description of the module $(\varphi_{fe},E_{fe})$ in terms of the conjugate module $F_e^*$ which, for computational purposes, has the advantage of not being built from balanced tensor products.
For each $e \in \Gamma^1$ let  $\Aa_e := C_c(\Sigma_e \times G_e \times \Sigma_e)$. We think of $\Aa_e$ as a dense $*$-subalgebra of $A_e$ with multiplication defined for each $a,\, b \in \Aa_e$ by
\[
	(ab)(\mu,h,\nu) = \sum_{\substack{\sigma \in \Sigma_e\\k \in G_e}} a(\mu,k,\sigma)b(\sigma,k^{-1}h,\nu).
\]
\begin{rmk}
	If $\Sigma_e \times G_e \times \Sigma_e$ is considered as a groupoid with product $(\mu, g, \nu) \cdot (\nu,h,\sigma) = (\mu, gh, \sigma)$, then $A_e$ is the full groupoid $C^*$-algebra of $\Sigma_e \times G_e \times \Sigma_e$, and the product defined above is just the convolution product on $\Aa_e$.
\end{rmk}

Now suppose that $fe \in \Gamma^2$ and consider the right Hilbert $A_e$-module
\[
	(F_e^*)^{\Sigma_f \times (\Sigma_e \setminus \Delta_{fe})} := \bigoplus_{\mu \in \Sigma_f} \bigoplus_{\nu \in \Sigma_e \setminus \Delta_{fe}} F_e^*
\]
with the standard right $A_e$-valued inner product and coordinate-wise right action of $A_e$. To keep track of the conjugate module structure we identify $\bigoplus_{\mu \in \Sigma_f} C_c(G_f) \subseteq F_f$ with
$C_c(\Sigma_f\times G_f)$, and $\bigoplus_{\mu \in \Sigma_e} C_c(G_e) \subseteq F_e^*$ with $  C_c(G_e \times \Sigma_e)$.
Let $\Ee_{fe}:=C_c(\Sigma_f\times (\Sigma_e \setminus \Delta_{fe})\times G_e \times \Sigma_e)$ which we consider as a dense subspace of $(F_e^*)^{\Sigma_f \times (\Sigma_e \setminus \Delta_{fe})}$.

\begin{lem}\label{lem:Eidentify}
	Let $\Gg  = (\Gamma,G)$ be a locally finite nonsingular graph of countable groups. Then for each $fe \in \Gamma^2$ the modules $E_{fe}$ and $(F_e^*)^{\Sigma_f \times (\Sigma_e \setminus \Delta_{fe})}$ are isomorphic as right Hilbert $A_e$-modules. In addition, if either $f \ne \ol{e}$ or $\alpha_{e}$ is not surjective, then there is a faithful non-degenerate left action $\psi_{fe}\colon A_f \to \End_{A_e} ((F_e^*)^{\Sigma_f \times (\Sigma_e \setminus \Delta_{fe})})$ satisfying
	\begin{equation}\label{eq:leftaction}
		(\psi_{fe}(a)\xi) (\mu,\nu,h,\sigma) = \sum_{\substack{ k \in G_{f} \\\rho \in \Sigma_f}} a(\mu,k^{-1},\rho)\, \xi(\rho, k \cdot \nu, c_{fe}(k,\nu)h,\sigma)
	\end{equation}
	for all $a \in \Aa_e$ and $\xi \in \Ee_{fe}$, which makes
	$(\varphi_{fe},E_{fe})$ and $(\psi_{fe},(F_e^*)^{\Sigma_f \times (\Sigma_e \setminus \Delta_{fe})})$ isomorphic as $A_f$--$A_e$-correspondences.
\end{lem}
\begin{proof}
	Define a linear map $\Phi \colon C_c(\Sigma_f \times G_f) \otimes_{C_c(G_f)} C_c((\Sigma_e \setminus \Delta_{fe}) \times G_e) \otimes_{C_c(G_e)} C_c(G_e \times \Sigma_e) \to \Ee_{fe}$ by
	\[
		\Phi(a \otimes \xi \otimes b)(\mu, \sigma, h, \nu) = \sum_{g \in G_f}\sum_{k \in G_e} a(\mu,g^{-1})\,\xi(g\cdot \sigma, c_{fe}(g,\sigma)k) b(k^{-1}h,\nu).
	\]
	Using the inner product on $E_{fe}$ we compute for $a \in C_c(\Sigma_f \times G_f)$, $\xi \in C_c((\Sigma_e \setminus \Delta_{fe}) \times G_e)$, and $b \in C_c(G_e \times \Sigma_e)$,
	\begin{align*}
		 & (a \otimes \xi \otimes b \mid a \otimes \xi \otimes b)_{\Aa_e} (\mu,h,\nu)                                 \\
		 & =(b \mid  (\xi \mid \ol{\varphi}_{fe}((a \mid a)_{C_c(G_f)})  \xi)_{C_c(G_e)} \cdot b)_{\Aa_e} (\mu,h,\nu) \\
		 & = \sum_{\substack{k \in G_e                                                                                \\g \in G_e}} \ol{b(k,\mu)} (\xi \mid \ol{\varphi}_{fe}((a \mid a)_{C_c(G_f)}) \xi)_{C_c(G_e)} (g) \,b(g^{-1}kh,\nu)\\
		 & = \sum_{\substack{k \in G_e                                                                                \\g \in G_e}} \sum_{\substack{\rho \in \Sigma_e \setminus \Delta_{fe}\\l \in G_e}} \ol{\xi(\mu , l )}\ol{b(k,\mu)}  \ol{\varphi}_{fe}((a \mid a)_{C_c(G_f)})\xi(\rho,lg) b(g^{-1}kh,\nu)\\
		 & = \sum_{\substack{k \in G_e                                                                                \\g \in G_e}} \sum_{\substack{\rho \in \Sigma_e \setminus \Delta_{fe}\\l \in G_e}}\sum_{m \in G_f} \ol{\xi(\rho , l )}\ol{b(k,\mu)} (a \mid a)_{C_c(G_f)}(m^{-1})\xi(m \cdot \rho,c_{fe}(m,\rho)lg) b(g^{-1}kh,\nu)\\
		 & = \sum_{\substack{k \in G_e                                                                                \\g \in G_e}} \sum_{\substack{\rho \in \Sigma_e \setminus \Delta_{fe}\\l \in G_e}}\sum_{\substack{m \in G_f}} \sum_{\substack{\sigma \in \Sigma_f\\r \in G_f}} \ol{a(\sigma,r)}\ol{\xi(\rho , l )}\ol{b(k,\mu)}a(\sigma,rm^{-1})\xi(m \cdot \rho,c_{fe}(m,\rho)lg) b(g^{-1}kh,\nu).
	\end{align*}
	Now make the following sequence of substitutions: $rm^{-1} \mapsto s$, $r \cdot \rho \mapsto \lambda$, $l \mapsto c_{fe}(r^{-1},\lambda)d$, $dg \mapsto t$, and $dk \mapsto u$. Upon performing these substitutions, the cocycle condition \eqref{eq:cocycle} implies that
	\begin{align*}
		 & (a \otimes \xi \otimes b \mid a \otimes \xi \otimes b)_{\Aa_e} (\mu,h,\nu)                  \\
		 & = \sum_{\substack{\sigma \in \Sigma_f                                                       \\\lambda \in \Sigma_e \setminus \Delta_{fe}}} \sum_{u \in G_e}
		\Big(\sum_{\substack{ r \in G_{f}                                                              \\d \in G_e}}
		\ol{a(\sigma,r) \xi(r^{-1} \cdot \lambda, c_{fe}(r^{-1},\lambda) d) b(d^{-1}u,\mu)} \Big)      \\
		 & \quad \times
		\Big(\sum_{\substack{ s \in G_{f}                                                              \\ t \in G_e}} a(\sigma,s) \xi(s^{-1} \cdot \lambda, c_{fe}(s^{-1},\lambda) t ) b(t^{-1}uh,\nu) \Big) \\
		 & =  \sum_{\substack{\sigma \in \Sigma_f                                                      \\\lambda \in \Sigma_e \setminus \Delta_{fe}}} \sum_{u \in G_e} \ol{\Phi(a \otimes \xi \otimes b)(\sigma, \lambda, u, \mu)} \Phi(a \otimes \xi \otimes b)(\sigma, \lambda, uh, \nu)\\
		 & = (\Phi(a \otimes \xi \otimes b) \mid \Phi(a \otimes \xi \otimes b))_{C_c(G_e)}(\mu,h,\nu).
	\end{align*}
	As such, $\Phi$ extends to an isometric linear map $\Phi\colon E_{fe} \to (F_e^*)^{\Sigma_f \times (\Sigma_e \setminus \Delta_{fe})}.$
	Let $\epsilon^f_{\mu,g}$ denote the point mass at $(\mu,g) \in \Sigma_f \times G_f$; let $\chi^{fe}_{\nu,h}$ denote the point mass at $h \in G_e$ in the $\nu$-th copy of $C^*(G_e)$ in $D_{fe}$; and let $\epsilon_{k,\sigma}^e$ denote the point mass at $(k,\sigma) \in G_e \times \Sigma_e$. Then $\Phi ( \epsilon_{\mu,1}^e \ox \chi^{fe}_{\sigma,1} \ox \epsilon_{h,\nu}^e)$ is the point mass at $(\mu,\nu,h,\sigma)$ in $\Ee_{fe}$. Since $\Ee_{fe}$ is spanned by point masses, $\Phi$ is surjective.

	It is routine to check that $\Phi$ respects the right action of $A_e$. By pulling the left action $\varphi_{fe}$ of $A_f$ on $E_{fe}$ over to a left action on $(F_e^*)^{\Sigma_f \times (\Sigma_e \setminus \Delta_{fe})}$ we see that \eqref{eq:leftaction} is satisfied, and as such $(\varphi_{fe},E_{fe})$ and $(\psi_{fe},(F_e^*)^{\Sigma_f \times (\Sigma_e \setminus \Delta_{fe})})$ are isomorphic.
\end{proof}

In the sequel we will always identify  $(\varphi_{fe},E_{fe})$ with $(\psi_{fe},(F_e^*)^{\Sigma_f \times (\Sigma_e \setminus \Delta_{fe})})$.

\begin{ntn}
	We introduce the following notation for point masses.
	For each $e \in \Gamma^1$ we write $\pointmv{e}{\mu}{g}{\nu}$ for the point mass at $(\mu,g,\nu) \in \Sigma_e \times G_e \times \Sigma_e$, considered as element of the dense $*$-subalgebra $\Aa_e = C_c(\Sigma_e \times G_e \times \Sigma_e)$ of $A_e$.
	For each $fe \in \Gamma^2$ let
	\begin{equation}
		X_{fe} = \Sigma_f \times (\Sigma_e \setminus \Delta_{fe} )\times G_e \times
		\Sigma_e
		\label{eq:X}
	\end{equation}
	so that $\Ee_{fe} = C_c(X_{fe})$.  We let $\chare{fe}{\mu}{\nu}{g}{\sigma}$ denote the point mass at $(\mu,\nu,g,\sigma) \in X_{fe}$.
\end{ntn}

\begin{rmk}
	For the reader familiar with \cite[Definition 2.4]{BMPST17}, it is possible to view elements of $X_{fe}$ as a triple consisting of a ``$\Gg$-path'' of length 2, a group element in $G_e$, and a ``$\Gg$-path'' of length 1. To be more precise, let
	$
		Y_{fe}  = \{ (g_1fg_2e,k,g_3e) \in \Gg^2 \times G_e \times \Gg^1\}.
	$
	Then the map $\pi \colon Y_{fe} \to X_{fe}$ given by $\pi(g_1fg_2e,k,g_3e) = (g_1,g_2,k,g_3)$ is a bijection, even in the case where $f = \ol{e}$.
\end{rmk}

We record some identities for point masses which can be deduced from the structures of $A$ and $(F_e^*)^{\Sigma_f \times (\Sigma_e \setminus \Delta_{fe})}$, together with Lemma~\ref{lem:Eidentify}.
With $\delta$ the Kronecker delta we have
\begin{align}
	\pointmv{f}{\mu}{g}{\nu} \pointmv{e}{a}{h}{b} & = \delta_{\nu f, a e} \, \pointmv{f}{\mu}{gh}{b} \label{eq:pmmult}                         \\
	(\chare{fe}{\mu}{\nu}{g}{\sigma} \mid \chare{pq}{a}{b}{h}{c})_A
	                                              & = \delta_{\mu f \nu e, apbq}\, \pointmv{e}{\sigma}{g^{-1}h}{c} \label{eq:pmip}             \\
	\chare{fe}{\mu}{\nu}{g}{\sigma}  \cdot \pointmv{t}{a}{h}{b}
	                                              & = \delta_{\sigma e, a t} \, \chare{fe}{\mu}{\nu}{gh}{b} \label{eq:pmright}, \text{ and}    \\
	\varphi(\pointmv{t}{a}{h}{b}) \chare{fe}{\mu}{\nu}{g}{\sigma}
	                                              & = \delta_{bt,\mu f} \, \chare{fe}{a}{h\cdot \nu}{c_{fe}(h,\nu)}{\sigma} \label{eq:pmleft}.
\end{align}
Together, \eqref{eq:pmip} and \eqref{eq:pmright} imply that
\begin{equation} \label{eq:pmcpt}
	\Theta_{\chare{fe}{\mu}{\nu}{g}{\sigma}, \chare{pq}{a}{b}{h}{c}} = \delta_{\sigma e, c q} \,	\Theta_{\chare{fe}{\mu}{\nu}{g}{\sigma}, \chare{pe}{a}{b}{h}{\sigma}}.
\end{equation}

We finish this section with a description of a frame for $E$, and consequently a description of $\varphi(a)$ in terms of rank-1 operators. Frames were introduced by Frank and Larson \cite{FL02} for Hilbert $C^*$-modules as a generalisation of orthonormal bases. We take the definition from \cite{RRS19} that \hl{frame for $E$} consists of a sequence $(u_i)_{i=1}^\infty$ in $E$ such that $\sum_{i=1}^{\infty}\Theta_{u_i,u_i}$ converges strictly to the identity operator on $E$.

\begin{lem}\label{lem:compacts}
	Let $\Gg  = (\Gamma,G)$ be a locally finite nonsingular graph of countable groups. Then
	\[
		(\chare{fe}{\mu}{\nu}{1}{1})_{fe \in \Gamma^2,\,\mu \in \Sigma_f,\, \nu \in \Sigma_e \setminus \Delta_{fe}}
	\]
	constitutes a frame for $E$. In particular, for each $f \in \Gamma^1$ and $a \in A_f$ we have
	\begin{equation}\label{eq:compacts}
		\varphi(a) = \sum_{\mu \in \Sigma_f} \sum_{\substack{\substack{r(e)= s(f)}\\\nu \in \Sigma_e \setminus \Delta_{fe}}}  \Theta_{\varphi(a)\chare{fe}{\mu}{\nu}{1}{1},\chare{fe}{\mu}{\nu}{1}{1}}.
	\end{equation}
\end{lem}
\begin{proof}
	Let $\epsilon_{g,\sigma}^e \in F_e^*$ denote the point mass at $(g,\sigma) \in G_e \times \Sigma_e$. Observe that for each $e \in \Gamma^1$ the right $A_e$-module $F_e^*$ admits a frame consisting of just $\epsilon_{1,1}^e$. Since Lemma~\ref{lem:Eidentify} implies that $E$ is a direct sum of copies of $F_e^*$, the point masses $\chare{fe}{\mu}{\nu}{1}{1}$ constitute a frame for $E$.

	The second statement follows from the first: if $(u_i)_{i=1}^\infty$ is a frame for $E$ and $a \in A$, then since $\varphi(a)$ is compact,
	$
		\varphi(a) = \varphi(a) \sum_{i=1}^\infty \Theta_{u_i,u_i} =  \sum_{i=1}^\infty \Theta_{\varphi(a)u_i,u_i}
	$
	with convergence in norm.
\end{proof}

\section{A Cuntz-Pimsner model for $C^*(\Gg)$}
\label{sec:cpmodel}

In this section we prove our first main result: that $C^*(\Gg)$ is isomorphic to the Cuntz-Pimsner algebra of the amplified graph of groups correspondence $(\varphi,E)$ associated to $\Gg$. Cuntz-Pimsner algebras can be considered a ``generalised crossed products by $\ZZ$'' and so in light of Theorem~\ref{thm:bmpst17iso}, we transform a crossed product by a potentially complicated group acting on the boundary of a tree to a generalised crossed product by $\ZZ$.

We recall from \cite{Kat04cor} that Cuntz-Pimsner algebras admit a \hl{gauge action}. That is an action of the circle group $\gamma \colon \TT \to \Aut(\Oo_E)$  such that $\gamma_z (i_E(\xi)) = z i_E(\xi)$ and $\gamma_z(i_A(a)) = i_A(a)$ for all $\xi \in E$ and $a \in A$, where $(i_A,i_E) \colon (\varphi,E) \to \Oo_E$ is the universal Cuntz-Pimsner covariant representation. Similarly, it is stated in \cite[Proposition 3.7]{BMPST17} that $C^*(\Gg)$ admits a gauge action $\gamma' \colon \TT \to \Aut(C^*(\Gg))$ satisfying $\gamma'_z(s_e) = zs_e$ and $\gamma'_z(u_{x,g}) = u_{x,g}$ for all $e \in \Gamma^1$, $x \in \Gamma^0$ and $g \in G_x$.

\begin{thm}\label{thm:isomorphism}
	Let $\Gg = (\Gamma,G)$ be a locally finite nonsingular graph of countable groups and let $(\varphi,E)$ be the associated amplified graph of groups correspondence. Then the Cuntz-Pimsner algebra $\Oo_{E}$ is gauge-equivariantly isomorphic to $C^*(\Gg)$.

\end{thm}

For the remainder of this section we fix a locally finite nonsingular graph of countable groups $\Gg = (\Gamma,G)$.
Recall that  $C^*(\Gg)$ is generated by a universal $\Gg$-family $\{u_{x},s_e \mid x \in \Gamma^0, e \in \Gamma^1\}$. For each $e \in \Gamma^1$ and $g \in G_{r(e)}$ write $s_{ge}$ for the partial isometry $u_{r(e),g}s_e$.
In order to prove Theorem~\ref{thm:isomorphism} we first construct a Cuntz-Pimsner covariant representation $(\pi,\psi)$ of the $C^*$-correspondence $E$ inside $C^*(\Gg)$. The universal property of Cuntz-Pimsner algebras then induces a $*$-homomorphism $\Psi \colon \Oo_{E} \to C^*(\Gg)$.

To simplify some of the computations in $C^*(\Gg)$ we have the following lemma.

\begin{lem} \label{lem:complem}
	Let $e,f \in \Gamma^1$ and suppose that $g \in \Sigma_e$ and $h \in \Sigma_f$. Then
	\begin{equation}\label{eq:orth}
		s_{ge}^*s_{hf} = \delta_{ge,hf} s_e^*s_e,
	\end{equation}
	and
	\begin{equation}\label{eq:proj}
		s_e^*s_e s_{hf} = \delta_{r(f),s(e)}(1-\delta_{hf,1\ol{e}}) s_{hf}.
	\end{equation}
\end{lem}
\begin{proof}
	The identities \ref{itm:G3} and \ref{itm:G4} imply that the projections $s_{ge}s_{ge}^*$ and $s_{hf}s_{hf}^*$ are orthogonal whenever $ge \ne hf$. It then follows that,
	\begin{align*}
		s_{ge}^*s_{hf} = 	s_{ge}^*s_{ge}s_{ge}^*s_{hf}s_{hf}^*s_{hf}
		= \delta_{ge,hf}  s_{ge}^*s_{ge} = \delta_{ge,hf} s_e^* u_{r(e),g}^* u_{r(e),g} s_e =\delta_{ge,hf} s_e^*s_e.
	\end{align*}
	For the second identity, we again compute using \ref{itm:G3} and \ref{itm:G4},
	\begin{align*}
		s_e^*s_e s_{hf} = 	\Big(\sum_{\substack{r(t)=s(e) \\ a \in \Sigma_t \setminus \Delta_{te}}}s_{at}s_{at}^* \Big) s_{hf} s_{hf}^* s_{hf} = \delta_{r(f),s(e)}(1-\delta_{hf,1\ol{e}}) s_{hf}. &\qedhere
	\end{align*}
\end{proof}

\begin{lem}\label{lem:univhom} Let $A=\bigoplus A_e$ be the algebra from Definition \ref{defn:AandE}.
	There is a $*$-homomorphism $\pi\colon A \to C^*(\Gg)$ satisfying
	\begin{equation}
		\pi (\epsilon^e_{\mu,g,\nu}) = s_{\mu e} u_{s(e),\alpha_{\ol{e}}(g)} s_{\nu e}^*
	\end{equation}
	for each $e \in \Gamma^1$ and $(\mu,g,\nu) \in \Sigma_e \times G_e \times \Sigma_e$.
\end{lem}
\begin{proof}
	For each $e \in \Gamma^1$ consider the $C^*$-subalgebra
	\[
		C_e := \overline{\spaan} \{s_{\mu e} u_{s(e),\alpha_{\ol{e}}(g)} s_{\nu e}^* \mid \mu,\nu \in \Sigma_e, \,g \in G_e\}
	\]
	of $C^*(\Gg)$. Note that \ref{itm:G2} implies $s_{\mu e} u_{s(e),\alpha_{\ol{e}}(g)} s_{\nu e}^* =s_{\mu e}s_{\nu e}^* u_{r(e),\alpha_{e}(g)}$. It follows from \eqref{eq:orth} and \ref{itm:G2} that
	\[
		(s_{\mu e} u_{s(e),\alpha_{\ol{e}}(g)} s_{\nu e}^*) (s_{a e} u_{s(e),\alpha_{\ol{e}}(h)} s_{b e}^*)= \delta_{\nu,a} s_{\mu e} u_{s(e),\alpha_{\ol{e}}(gh)} s_{b e}^*,
	\]
	and $(s_{\mu e} u_{s(e),\alpha_{\ol{e}}(g)} s_{\nu e}^*)^* = s_{\nu e}  u_{s(e),\alpha_{\ol{e}}(g^{-1})} s_{\mu e}^*$.
	In particular, the elements $\{s_{\mu e}s_{\nu e}^* \mid \mu,\,\nu \in \Sigma_e\}$ form a system of matrix units in $C_e$. Hence, there is a $*$-homomorphism $\theta \colon M_{\Sigma_e}(\CC) \to \spaan \{s_{\mu e}s_{\nu e}^* \mid \mu,\,\nu \in \Sigma_e\}$ satisfying $\theta(\pointmv{e}{\mu}{1}{\nu})= s_{\mu e}s_{\nu e}^*$. The map
	and $g \mapsto \sum_{\mu \in \Sigma_e} s_{\mu e}s_{\mu e} u_{r(e),\alpha_e(g)}$ defines a unitary representation of $G_e$ in $C_e$ which commutes with the range of $\theta$. Since $A_e \cong M_{|\Sigma_e|}(\CC) \ox C^*(G_e)$, there is a $*$-homomorphism $\pi_e\colon A_e \to C_e$ satisfying $\pi_e(\pointmv{e}{\mu}{g}{\nu}) = s_{\mu e} u_{s(e),\alpha_{\ol{e}}(g)} s_{\nu e}^*$.

	The identity \eqref{eq:orth} implies that  $\pi_e(A_e)\pi_f(A_f) = \pi_f(A_f)\pi_e(A_e) = 0$ for $e \ne f$. As such, $\pi := \oplus_{e \in \Gamma^1} \pi_e$ defines a $*$-homomorphism from $A$ to  $C^*(\Gg)$ such that $\pi (\epsilon^e_{\mu,g,\nu}) = s_{\mu e} u_{s(e),\alpha_{\ol{e}}(g)} s_{\nu e}^*$ for all $e \in \Gamma^1$, $\mu,\nu \in \Sigma_e$, and $g \in G_e$.
\end{proof}

\begin{lem}\label{lem:univlinearmap}
	There is a norm-decreasing linear map $\psi \colon E \to C^*(\Gg)$ satisfying
	\begin{equation}
		\psi(\chare{fe}{\mu}{\nu}{g}{\sigma}) = s_{\mu f}s_{\nu e} \us{e}{g} s_{\sigma e}^*
	\end{equation}
	for all $fe \in \Gamma^2$ and $(\mu,\nu,g,\sigma) \in X_{fe}$. Moreover, $\psi^*(\xi)\psi(\eta) = \pi((\xi \mid \eta)_A)$ for all $\xi,\,\eta \in E$.
\end{lem}
\begin{proof}
	For each $fe \in \Gamma^2$ define $\psi_{fe} \colon C_c(X_{fe}) \to C^*(\Gg)$ by $\psi_{fe}(\chare{fe}{\mu}{\nu}{g}{\sigma}) = s_{\mu f}s_{\nu e} \us{e}{g} s_{\sigma e}^*$. Let $\psi :=  \oplus_{fe \in \Gamma^2} \psi_{fe}$ denote the induced map from the algebraic direct sum $\bigoplus_{fe \in \Gamma^2} C_c(X_{fe})$ to $C^*(\Gg)$.
	Fix $fe,\,tr \in \Gamma^2$, $(\mu,\nu,g,\sigma) \in X_{fe}$, and $(a,b,h,c) \in X_{tr}$. It follows from Lemma~\ref{lem:complem}, \ref{itm:G2}, and \eqref{eq:pmip} that
	\begin{align*}
		\psi(\chare{fe}{\mu}{\nu}{g}{\sigma})^*\psi(\chare{tr}{a}{b}{h}{c})
		 & =s_{\sigma e} \us{e}{g^{-1}} s_{\nu e}^* s_{\mu f}^*s_{at} s_{br} \us{r}{h} s_{cr}^*                                  \\
		 & =\delta_{\mu f,at} s_{\sigma e} \us{e}{g^{-1}} s_{\nu e}^* s_{ f}^*s_{f} s_{br} \us{r}{h} s_{cr}^*                    \\
		 & =\delta_{\mu f,at} (1-\delta_{br,1\ol{f}}) s_{\sigma e} \us{e}{g^{-1}} s_{\nu e}^* s_{br} \us{r}{h} s_{cr}^*          \\
		 & =\delta_{\mu f,at} (1-\delta_{br,1\ol{f}}) \delta_{\nu e, br} s_{\sigma e} \us{e}{g^{-1}} s_e^*s_e \us{r}{h} s_{cr}^* \\
		 & =\delta_{\mu f,at} \delta_{\nu e, br} (1-\delta_{\nu e,1\ol{f}})  s_{\sigma e} \us{e}{g^{-1}h}  s_{ce}^*              \\
		 & =\delta_{\mu f,at} \delta_{\nu e, br} (1-\delta_{\nu e,1\ol{f}}) \pi(\pointmv{e}{\sigma}{g^{-1}h}{c})                 \\
		 & = \pi((\chare{fe}{\mu}{\nu}{g}{\sigma}\mid \chare{tr}{a}{b}{h}{c})_A).
	\end{align*}
	Consequently, $\psi(\xi)^* \psi(\eta) = \pi((\xi\mid \eta)_A)$ for all $\xi, \eta \in \bigoplus_{fe \in \Gamma^2} C_c(X_{fe}) $. This in turn gives $\|\psi(\xi)\|^2 = \|\pi((\xi \mid \xi)_A)\| \le \|(\xi \mid \xi)_A\| = \|\xi\|^2$, so that $\psi$ extends to a bounded linear map $\psi\colon E \to C^*(\Gg)$ with the property that $\psi(\xi)^*\psi(\eta) = \pi((\xi\mid \eta)_A)$ for all $\xi,\,\eta \in E$.
\end{proof}

\begin{prop}\label{prop:cpcov}
	The pair $(\pi,\psi)$ defines a Cuntz-Pimsner covariant representation of $E$ in $C^*(\Gg)$. In particular, if  $(i_A,i_{E}) \colon (\varphi,E) \to \Oo_{E}$ is the universal Cuntz-Pimsner covariant representation, then
	there is an induced $*$-homomorphism $\Psi \colon \Oo_{E} \to C^*(\Gg)$ such that $\pi(a) =(\Psi \circ i_{A})(a)$ and $\psi(\xi) = (\Psi \circ i_{E})(\xi)$ for all $a \in A$ and $\xi \in E$.
\end{prop}

\begin{proof}
	Given Lemma~\ref{lem:univhom} and Lemma~\ref{lem:univlinearmap} all that remains is to show that $(\pi,\psi)$ respects both the right and left actions of $A$ on $E$, and satisfies Cuntz-Pimsner covariance.
	Using Lemma~\ref{lem:complem}, \ref{itm:G2}, and \eqref{eq:pmright} we compute
	\begin{align*}
		\psi(\chare{fe}{\mu}{\nu}{g}{\sigma}) \pi(\pointmv{t}{a}{h}{b})
		 & = (s_{\mu f} s_{\nu e} \us{e}{g} s_{\sigma e}^*)(s_{at}u_{s(t),\alpha_{\ol{t}}(h)} s_{bt}^*) \\
		 & = \delta_{\sigma e, at} s_{\mu f} s_{\nu e} \us{e}{gh}  s_{be}^*                             \\
		 & = \delta_{\sigma e, at} \psi(\chare{fe}{\mu}{\nu}{gh}{b})                                    \\
		 & = \psi(\chare{fe}{\mu}{\nu}{g}{\sigma} \cdot \pointmv{t}{a}{h}{b}).
	\end{align*}
	It follows that $\psi(\xi)\pi(a) = \psi(\xi \cdot a)$ for all $\xi \in E$ and $a \in A$.
	For the left action we use Lemma~\ref{lem:carrying} and \eqref{eq:pmleft} to see that
	\begin{align*}
		\pi(\pointmv{t}{a}{h}{b})  \psi(\chare{fe}{\mu}{\nu}{g}{\sigma})
		 & = (s_{at}u_{s(t),\alpha_{\ol{t}}(h)} s_{bt}^*) (s_{\mu f} s_{\nu e} \us{e}{g} s_{\sigma e}^*)   \\
		 & =\delta_{bt,\mu f} s_{af}u_{s(f),\alpha_{\ol{f}}(h)}s_f^*s_f s_{\nu e} \us{e}{g} s_{\sigma e}^* \\
		 & =\delta_{bt,\mu f} s_{af}u_{s(f),\alpha_{\ol{f}}(h)} s_{\nu e} \us{e}{g} s_{\sigma e}^*         \\
		 & =\delta_{bt,\mu f} s_{af} s_{(h \cdot \nu) e} \us{e}{c_{fe}(h,\nu)g} s_{\sigma e}^*             \\
		 & =\delta_{bt,\mu f} \psi(\chare{fe}{\mu}{h \cdot \nu}{c_{fe}(h,\nu)g}{\sigma})                   \\
		 & =  \psi\big(\varphi(\pointmv{t}{a}{h}{b})  \chare{fe}{\mu}{\nu}{g}{\sigma}\big).
	\end{align*}
	Hence, $\pi(a) \psi(\xi) = \psi(\varphi(a) \xi)$ for all $a \in A$ and $\xi \in E$.
	Consequently, $(\pi,\psi)$ is a representation of the $C^*$-correspondence $(\varphi,E)$ in $C^*(\Gg)$.

	We now claim that $(\pi,\psi)$ is Cuntz-Pimsner covariant.  Using \ref{itm:G4} and Lemma~\ref{lem:compacts} we compute
	\begin{align*}
		(\psi^{(1)} \circ \varphi)(1_{A_f})
		 & = \sum_{\mu \in \Sigma_f} \sum_{\substack{\substack{r(e)= s(f)}               \\\nu \in \Sigma_e \setminus \Delta_{fe}}}  \psi(\chare{fe}{\mu}{\nu}{1}{1}) \psi(\chare{fe}{\mu}{\nu}{1}{1})^*\\
		 & =\sum_{\mu \in \Sigma_f} \sum_{\substack{\substack{r(e)= s(f)}                \\\nu \in \Sigma_e \setminus \Delta_{fe}}}   s_{\mu f}s_{\nu e} s_{e}^* s_e s_{\nu e}^* s_{\mu f}^* \\
		 & =\sum_{\mu \in \Sigma_f} s_{\mu f} \Big(\sum_{\substack{\substack{r(e)= s(f)} \\\nu \in \Sigma_e \setminus \Delta_{fe}}}   s_{\nu e} s_{\nu e}^*\Big) s_{\mu f}^* \\
		 & =\sum_{\substack{\mu \in \Sigma_f}}  s_{\mu f}s_f^*s_f s_{\mu f}^*            \\
		 & = \sum_{\substack{\mu \in \Sigma_f}}  s_{\mu f} s_{\mu f}^*                   \\
		 & = \pi \Big(\sum_{\mu \in \Sigma_f} \pointmv{e}{\mu}{1}{\mu}  \Big)            \\
		 & = \pi (1_{A_f}).
	\end{align*}
	With  \eqref{eq:compacts} the preceding computation can be modified to show that $(\psi^{(1)} \circ \varphi)(a) = \pi(a)$ for all $a \in \Aa_f$ and therefore all $a \in A_f$. Since each $a \in A$ can be approximated by a finite sum of $a_e \in A_e$, it follows that $(\psi^{(1)} \circ \varphi)(a) = \pi(a)$ for all $a \in A$.
	The universal property of $\Oo_{E}$ now induces a $*$-homomorphism $\Psi \colon \Oo_{E} \to C^*(\Gg)$ such that $\pi(a) =(\Psi \circ i_{A})(a)$ and $\psi(\xi) = (\Psi \circ i_{E})(\xi)$ for all $a \in A$ and $\xi \in E$.
\end{proof}

To see that $\Psi$ is an isomorphism we construct an inverse $\Phi$.
To this end, we identify a $\Gg$-family $\{V_{x},T_e \mid x \in \Gamma^0, e \in \Gamma^1\}$ within $\Oo_E$ and then use the universal property of $C^*(\Gg)$ to obtain a $*$-homomorphism $\Phi \colon C^*(\Gg) \to \Oo_E$ that is inverse to $\Psi$.

\begin{lem}\label{lem:partialiso}
	Let $(i_A,i_E) \colon E \to \Oo_E$ denote the universal Cuntz-Pimsner covariant representation. For each $f \in \Gamma^1$ let
	\begin{equation}
		T_{f} := \sum_{\substack{\substack{r(e)= s(f)}\\\mu \in \Sigma_e \setminus \Delta_{fe}}} i_E\big(\chare{fe}{1}{\mu}{1}{\mu}\big).
	\end{equation}
	Then $T_f$ is a partial isometry in $\Oo_E$ with source and range projections given by
	\begin{equation}\label{eq:rangesource}
		T_{f}^*T_{f} =  \sum_{\substack{\substack{r(e)= s(f)}\\\mu \in \Sigma_e \setminus \Delta_{fe}}}i_A(\pointmv{e}{\mu}{1}{\mu})
		\qquad\text{and}\qquad
		T_fT_f^* = i_A(\pointmv{f}{1}{1}{1}).
	\end{equation}
\end{lem}
\begin{proof}
	The first identity of \eqref{eq:rangesource} follows from \eqref{eq:pmip} as for all $e,t \in r^{-1}(s(f))$, $\mu \in \Sigma_e$ and $\nu \in \Sigma_t$,
	\[
		i_E(\chare{fe}{1}{\mu}{1}{\mu})^*	i_E(\chare{ft}{1}{\nu}{1}{\nu}) = i_A((\chare{fe}{1}{\mu}{1}{\mu} \mid \chare{ft}{1}{\nu}{1}{\nu} )_A) = \delta_{\mu e,\nu t}\, i_A(\pointmv{e}{\mu}{1}{\mu}).
	\]
	For the second identity we use \eqref{eq:pmcpt} to see that for all $e,\,t \in r^{-1}(s(f))$, $\mu \in \Sigma_e$ and $\nu \in \Sigma_t$,
	\[
		i_E(\chare{fe}{1}{\mu}{1}{\mu})	i_E(\chare{ft}{1}{\nu}{1}{\nu})^* = i_E^{(1)}\Big(\Theta_{\chare{fe}{1}{\mu}{1}{\mu},\chare{ft}{1}{\nu}{1}{\nu}}\Big)
		= \delta_{\mu e,\nu t} \, i_E^{(1)} \Big(\Theta_{\chare{fe}{1}{\mu}{1}{\mu},\chare{fe}{1}{\mu}{1}{\mu}}\Big).
	\]
	It now follows from Cuntz-Pimsner covariance of $(i_A,i_E)$, \eqref{eq:pmleft}, and \eqref{eq:compacts} that
	\begin{align*}
		T_{f}T_{f}^* & = \sum_{\substack{\substack{r(e)= s(f)}                        \\\mu \in \Sigma_e \setminus \Delta_{fe}}} \Big(\Theta_{\chare{fe}{1}{\mu}{1}{\mu},\chare{fe}{1}{\mu}{1}{\mu}}\Big)\\
		             & = \sum_{\nu \in \Sigma_f}\sum_{\substack{\substack{r(e)= s(f)} \\\mu \in \Sigma_e \setminus \Delta_{fe}}} \Big(\Theta_{\varphi_{fe}(\pointmv{f}{1}{1}{1})\chare{fe}{1}{\nu}{1}{\mu},\chare{fe}{1}{\nu}{1}{\mu}}\Big)\\
		             & = (i_E^{(1)} \circ \varphi)(\pointmv{f}{1}{1}{1})              \\
		             & = i_A(\pointmv{f}{1}{1}{1}). \qedhere
	\end{align*}
\end{proof}

Our aim now is to establish partial unitary representations of $G_x$ in $\Oo_E$ for each $x \in \Gamma^0$. To accomplish this, let $x \in \Gamma^0$ and $g \in G_x$ and consider the element of $\Oo_E$ defined by,
\begin{equation}
	V_{x,g}
	:= \sum_{\substack{r(f)=x\\\mu \in \Sigma_f}} i_A\big(\pointmv{f}{g \cdot \mu}{c_{f}(g,\mu)}{\mu}\big).
	\label{eq:vee}
\end{equation}
Note that when $g = 1$ we have $V_{x,1} = \sum_{r(f) = x} i_A(1_{A_f})$.

\begin{lem}\label{lem:partialunit}
	For each $x \in \Gamma^0$ the map $G_x \ni g \mapsto V_{x,g}$ defines a partial unitary representation of $G_x$ in $\Oo_E$.
\end{lem}

\begin{proof}
	Fix $x \in \Gamma^0$ and  $g,h \in G_x$. The cocycle condition \eqref{eq:cocycle} implies that for all $e,f \in r^{-1}(x)$, $\mu \in \Sigma_f$ and $\nu \in \Sigma_e$,
	\[
		\pointmv{f}{g \cdot \mu}{c_f(g,\mu)}{\mu} \pointmv{e}{h \cdot \nu}{c_e(h,\nu)}{\nu} = \delta_{f,e} \delta_{\mu, h \cdot \nu} \pointmv{f}{g\cdot (h \cdot \nu)}{c_{f}(g,h \cdot \nu)c_f(h,\nu)}{\nu} = \delta_{f,e} \delta_{\mu, h \cdot \nu} \pointmv{f}{(gh) \cdot \nu}{c_f(gh,\nu)}{\nu}.
	\]
	As such, $V_{x,g}V_{x,h} = V_{x,gh}$.  Making the substitution $\mu = g^{-1} \cdot \nu$ at the second equality, and using the cocycle condition \eqref{eq:cocycle} to see that $c_f(g,g^{-1}\cdot \nu)c_f(g^{-1},\nu) = 1$, we have
	\begin{align}\label{eq:unitadjoint}
		\begin{split}
			V_{x,g}^*
			&= \sum_{\substack{r(f)=x\\\mu \in \Sigma_f}} i_A(\pointmv{f}{\mu}{c_f(g,\mu)^{-1}}{g \cdot \mu}) \\
			&= \sum_{\substack{r(f)=x\\\nu \in \Sigma_f}} i_A(\pointmv{f}{ g^{-1} \cdot \nu}{c_f(g,g^{-1}\cdot \nu)^{-1}}{\nu})\\
			&= \sum_{\substack{r(f)=x\\\nu \in \Sigma_f}} i_A(\pointmv{f}{ g^{-1}\cdot \nu}{c_f(g^{-1},\nu)}{ \nu})  \\
			&= V_{x,g^{-1}}.
		\end{split}
	\end{align}
	Consequently, $V_{x,g}$ is a partial unitary and $g \mapsto V_{x,g}$ is a partial unitary representation of $G_x$.
\end{proof}

\begin{prop}
	The collection $\{V_{x},T_e \mid x \in \Gamma^0, e \in \Gamma^1\}$ defines a $\Gg$-family in $\Oo_E$. In particular, there is a unique $*$-homomorphism $\Phi\colon C^*(\Gg) \to \Oo_E$ such that $\Phi(s_e) = T_e$ for all $e \in \Gamma^1$, and $\Phi(u_{x,g}) = V_{x,g}$ for all $x \in \Gamma^0$ and $g \in G_x$.
\end{prop}

\begin{proof}
	We will show that $\{V_{x},T_e \mid x \in \Gamma^0, e \in \Gamma^1\}$ satisfies \ref{itm:G1}--\ref{itm:G4} and then use the universal property of $C^*(\Gg)$ to give $\Phi$. The identity \ref{itm:G1} is follows immediately from direct sum decomposition of $A$: if $x \ne y \in \Gamma^0$ and $e,f \in \Gamma^1$ are such that $r(e) = x$ and $r(f) = y$ then $\pointmv{e}{\mu}{1}{\mu}\pointmv{f}{\nu}{1}{\nu} = 0$ for any $\mu \in \Sigma_e$ and $\nu \in \Sigma_f$. The identity \eqref{eq:rangesource} yields \ref{itm:G3}.
	Both \eqref{eq:rangesource} and \eqref{eq:unitadjoint} imply that for $f \in \Gamma^1$ and $\mu \in \Sigma_f$,
	\begin{align*}
		V_{r(f),\mu}T_f T_f^* V_{r(f),\mu}^* & =
		\sum_{\substack{r(e)=r(f)                                                                                                \\\nu \in \Sigma_e}}\sum_{\substack{r(t)=r(f)\\\sigma \in \Sigma_t}} i_A(\pointmv{e}{\mu \cdot \nu}{c_e(\mu,\nu)}{\nu}) i_A(\pointmv{f}{1}{1}{1}) i_A(\pointmv{t} {\sigma}{c_e(\mu,\sigma)^{-1}}{\mu \cdot \sigma})\\
		                                     & =i_A(\pointmv{f}{\mu}{1}{1}) i_A(\pointmv{f}{1}{1}{1})i_A(\pointmv{f}{1}{1}{\mu}) \\
		                                     & = i_A(\pointmv{f}{\mu}{1}{\mu}).
	\end{align*}
	The identity \ref{itm:G4} now follows from \eqref{eq:rangesource}.

	We now claim that \ref{itm:G2} is satisfied. For each $f \in \Gamma^1$  and $g \in G_f$,
	\begin{align*}
		V_{r(f),\alpha_f(g)} T_f & = \sum_{\substack{r(e)=r(f)   \\\nu \in \Sigma_e}}
		\sum_{\substack{r(t) = s(f)                              \\ \mu \in \Sigma_t \setminus \Delta_{ft}}}
		i_A(\pointmv{e}{\alpha_f(g) \cdot \nu}{c_{e}(\alpha_f(g),\nu) }{\nu})
		i_E(\chare{ft}{1}{\mu}{1}{\mu})                          \\
		                         & = \sum_{\substack{r(t) = s(f) \\ \mu \in \Sigma_t \setminus \Delta_{ft}}}	i_E(\varphi(\pointmv{f}{\alpha_f(g) \cdot 1}{c_{f}(\alpha_f(g),1)}{1})\chare{ft}{1}{\mu}{1}{\mu})\\
		                         & = \sum_{\substack{r(t) = s(f) \\ \mu \in \Sigma_t \setminus \Delta_{ft}}}	i_E(\varphi(\pointmv{f}{1}{g}{1})\chare{ft}{1}{\mu}{1}{\mu})\\
		                         & = \sum_{\substack{r(t) = s(f) \\ \mu \in \Sigma_t \setminus \Delta_{ft}}}
		i_E(\chare{ft}{1}{\alpha_{\ol{f}}(g) \cdot \mu}{c_{ft}(g,\mu)}{\mu}),
	\end{align*}
	where we have used the fact that $\alpha_f(g) \cdot 1 = 1$ and $c_{f}(\alpha_f(g),\nu) = g$.
	On the other hand,
	\begin{align*}
		T_f 	V_{s(f),\alpha_{\ol{f}}(g)}
		 & =\sum_{\substack{r(t) = s(f)                                                                                                                       \\ \mu \in \Sigma_t \setminus \Delta_{ft}}} \sum_{\substack{r(e)=r(f)\\\nu \in \Sigma_e}}
		i_E(\chare{ft}{1}{\mu}{1}{\mu})	i_A(\pointmv{e}{\alpha_{\ol{f}}(g) \cdot \nu}{c_{e}(\alpha_{\ol{f}}(g),\nu) }{\nu})                                    \\
		 & =\sum_{\substack{r(t) = s(f)                                                                                                                       \\ \alpha_{\ol{f}}(g) \cdot \nu \in \Sigma_t \setminus \Delta_{ft}}}
		i_E(\chare{ft}{1}{\alpha_{\ol{f}}(g) \cdot \nu}{1}{\alpha_{\ol{f}}(g) \cdot \nu} \cdot \pointmv{t}{\alpha_{\ol{f}}(g) \cdot \nu}{c_{ft}(g,\nu)}{\nu}) \\
		 & =
		\sum_{\substack{r(t) = s(f)                                                                                                                           \\ \alpha_{\ol{f}}(g) \cdot \nu \in \Sigma_t \setminus \Delta_{ft}}}
		i_E(\chare{ft}{1}{\alpha_{\ol{f}}(g) \cdot \nu}{c_{ft}(g,\nu)}{\nu}).                                                                                 \\
	\end{align*}
	After observing that $\alpha_{\ol{f}}(g) \cdot \nu \in \Sigma_t \setminus \Delta_{ft}$ if and only if $\nu \in \Sigma_t \setminus \Delta_{ft}$ we see that \ref{itm:G2} is satisfied. The existence of $\Phi$ now follows immediately from the universal property of $C^*(\Gg)$.
\end{proof}

We can now prove the main result of this section.
\begin{proof}[Proof of Theorem {\ref{thm:isomorphism}}] 	We claim that $\Psi$ and $\Phi$ are mutually inverse isomorphisms. Recall that $\Psi \circ i_A = \pi$ and $\Psi \circ i_E = \psi$.
	It follows that for each $f \in \Gamma^1$ we have,
	\[
		(\Psi \circ \Phi) (s_f) = \Psi(T_{f}) = {\sum_{\substack{r(e) = s(f) \\ \mu \in \Sigma_e \setminus \Delta_{fe}} }} s_f s_{\mu e} s_{\mu e}^*= s_f s_f^* s_f = s_f,
	\]
	and for each $x \in \Gamma^0$ and $g \in G_x$,
	\begin{align*}
		(\Psi \circ \Phi) (u_{x,g}) & = \Psi(V_{x,g})
		= \sum_{\substack{r(e)=x                                            \\\mu \in \Sigma_f}} s_{(g \cdot \mu)e}u_{s(e),c_e(g,\mu)} s_{\mu e}
		= \sum_{\substack{r(e)=x                                            \\\mu \in \Sigma_f}} u_{r(e),g} s_{\mu e} s_{\mu e}^* \\
		                            & = u_{x,g} \Big(\sum_{\substack{r(e)=x \\\mu \in \Sigma_f}}    s_{\mu e} s_{\mu e}^*\Big)
		= u_{x,g}.
	\end{align*}
	Now suppose that $(\mu,g,\nu) \in \Sigma_e \times G_e \times \Sigma_e$. Since $\{T_e \mid e \in \Gamma^1\} \cup \{V_x \mid x \in \Gamma^0\}$ forms a $\Gg$-family,
	\begin{align*}
		(\Phi \circ \Psi)(i_A(\pointmv{e}{\mu}{g}{\nu})) & = \Phi (s_{\mu e}u_{s(e), \alpha_{\ol{e}} (g)} s_{\nu e}^* )
		= T_{\mu e} V_{s(e),\alpha_{\ol{e}}(g)} T_{\nu e}^*
		= V_{r(e),\mu\alpha_e(g)}T_{e} T_{e}^*V_{r(e),\nu}^*                                                                   \\
		                                                 & = V_{r(e),\mu\alpha_e(g)} i_A(\pointmv{e}{1}{1}{1}) V_{r(e),\nu}^*,
	\end{align*}
	where we have used \eqref{eq:rangesource} for the last equality. Now using \eqref{eq:unitadjoint} we see that,
	\begin{align*}
		V_{r(e),\mu\alpha_e(g)} i_A(\pointmv{e}{1}{1}{1}) V_{r(e),\nu}^*
		 & =  \sum_{\substack{r(f)=r(e)                                                                     \\\sigma \in \Sigma_f}} \sum_{\substack{r(t)=r(e)\\\rho \in \Sigma_t}} i_A(\pointmv{f}{\mu\alpha_e(g) \cdot \sigma}{c_{f}(\mu\alpha_e(g),\sigma)}{\sigma}) i_A(\pointmv{e}{1}{1}{1})  i_A(\pointmv{t}{\rho}{c_t(\nu,\rho)^{-1}}{\nu \cdot \rho})\\
		 & = i_A(\pointmv{e}{\mu \alpha_e(g) \cdot 1}{c_e(\mu \alpha_e(g),1) c_e(\nu,1)^{-1}}{\nu \cdot 1}) \\
		 & = i_A(\pointmv{e}{\mu}{g}{\nu}).
	\end{align*}
	Now suppose that $(\mu,\nu,g,\sigma) \in X_{fe}$. Then,
	\begin{align*}
		(\Phi \circ \Psi)(i_E(\chare{fe}{\mu}{\nu}{g}{\sigma})) = \Phi (s_{\mu f} s_{\nu e} u_{s(e),\alpha_{\ol{e}}(g)} s_{\sigma e}^*)
		= T_{\mu f} T_{
				\nu e}V_{s(e),\alpha_{\ol{e}}(g)} T_{\sigma e}^*.
	\end{align*}
	Using the computation of $(\Phi \circ \Psi)(i_A(\pointmv{e}{\mu}{g}{\nu}))$ above we find that,
	\begin{align*}
		T_{\mu f} T_{
				\nu e}V_{s(e),\alpha_{\ol{e}}(g)} T_{\sigma e}^*
		 & =	T_{\mu f}V_{r(e),\nu\alpha_e(g)} T_e T_{e}^* V_{r(e),\sigma}^*                   \\
		 & = T_{\mu f} i_A(\pointmv{e}{\nu}{g}{\sigma})                                      \\
		 & = V_{r(f),\mu} \sum_{\substack{\substack{r(t)= s(f)}                              \\\rho \in \Sigma_t \setminus \Delta_{ft}}} i_E(\chare{ft}{1}{\rho}{1}{\rho}) i_A(\pointmv{e}{\nu}{g}{\sigma})\\
		 & = V_{r(f),\mu} i_E(\chare{fe}{1}{\nu}{g}{\sigma})                                 \\
		 & =\sum_{\substack{r(t)=r(e)                                                        \\\rho \in \Sigma_t}} i_A(\pointmv{t}{\mu \cdot \rho}{c_{f}(\mu,\rho)}{\rho}) i_E(\chare{fe}{1}{\nu}{g}{\sigma})\\
		 & = i_A(\pointme{f}{\mu \cdot 1}{c_f(\mu,1)}{1}) i_E(\chare{fe}{1}{\nu}{g}{\sigma}) \\
		 & =  i_E(\chare{fe}{\mu}{\nu}{g}{\sigma}).
	\end{align*}
	It now follows that $\Phi$ and $\Psi$ are mutually inverse isomorphisms. Gauge-equivariance follows from the definition of either $\Phi$ or $\Psi$.
\end{proof}

%

\section{$\Oo_D$ as an Exel-Pardo algebra}
\label{sec:ExelPardo}

We now make a return to the graph of groups correspondence $(\ol{\varphi},D)$ and its Cuntz-Pimsner algebra $\Oo_D$. In many ways, $\Oo_D$ is more elementary than $\Oo_E$, and since the algebras are Morita equivalent by Proposition~\ref{prop:DandEmoritaeq}, for many purposes we can instead consider $\Oo_D$. Cuntz-Pimsner algebras often admit a description in terms of generators and relations, and $\Oo_D$ is no exception.

\begin{prop}\label{prop:Dgenrel}
	Let $\Gg = (\Gamma,G)$ be a locally finite nonsingular graph of countable groups. Then $\Oo_D$ is the universal $C^*$-algebra generated by a collection of partial isometries
	$\{s_{f\mu e} \mid fe \in \Gamma^2 ,\, \mu \in \Sigma_e \setminus \Delta_{fe} \}$ and a
	family of partial unitary representations $u_{e} \colon g \mapsto u_{e,g}$ of $G_e$ for each $e \in \Gamma^1$ satisfying the relations:
	\begin{enumerate}[label={ (D\arabic*)},align=left]
		\setlength\itemsep{0.7em}
		\item \label{itm:D1}
		      $u_{e,1} u_{f,1} = 0$ for each $e,\,f \in \Gamma^1$ with $e \ne f$;
		\item \label{itm:D2}
		      $u_{f,g} s_{f\mu e} = s_{f(\alpha_{\ol{f}}(g) \cdot \mu)e} u_{e,c_{fe}(g,\mu)}$ for each $fe \in \Gamma^2$ and $g \in G_f$;
		\item \label{itm:D3}
		      $ s_{f\mu e}^*s_{f\mu e} = u_{e,1}$ for all $fe \in \Gamma^2$ and $\mu \in \Sigma_e \setminus \Delta_{fe}$;
		\item \label{itm:D4}
		      $\displaystyle \sum_{\substack{r(e) = s(f) \\ \mu \in \Sigma_{e} \setminus \Delta_{fe}}} s_{f\mu e} s_{f\mu e}^* = u_{f,1}$ for all $f \in \Gamma^1$.
	\end{enumerate}
\end{prop}

\begin{proof}
	Let $(i_B,i_D)$ denote the universal Cuntz-Pimsner covariant representation of $(\ol{\varphi},D)$  in $\Oo_D$.
	For each $e \in \Gamma^1$, let $\epsilon_g^e \in C_c(G_e) \subseteq C^*(G_e)$ denote the point mass at $g \in G_e$. For each $fe \in \Gamma^2$ and $\mu \in \Sigma_e \setminus \Delta_{fe}$ let $\chi_{\mu}^{fe}$ denote the identity for the $\mu$-th copy of $C^*(G_e)$ in the direct sum $D_{fe}$. Then set
	\[
		s_{ f\mu e} := i_D(\chi_{\mu}^{fe}) \quad\text{ and } \quad u_{e,g} := i_B(\epsilon_g^e).
	\]
	Using the fact that $(i_B,i_D)$ is a Cuntz-Pimsner covariant representation of $(\ol{\varphi}, D)$ it is straightforward to check that each $s_{ f\mu e}$ is a partial isometry, each $u_{e} \colon g \mapsto u_{e,g}$ is a partial unitary representation of $G_e$ in $D$, and the
	relations \ref{itm:D1}--\ref{itm:D4} are satisfied. Since the module $D = \ol{\spaan} \{ \chi_{\mu}^{fe} \cdot a \mid fe \in \Gamma^2,\, \mu \in \Sigma_e \setminus \Delta_{fe},\, a \in C^*(G_e)\}$, it follows that $\Oo_D$ is generated by
	$
		\{ s_{f \mu g } \mid fe \in \Gamma^2,\, \mu \in \Sigma_e \setminus \Delta_{fe}\} \cup \{ u_{e,g} \mid e \in \Gamma^1,\, g \in G_e\}.
	$

	Now let $C$ be a $C^*$-algebra generated by  partial isometries $\{t_{f\mu e} \mid fe \in \Gamma^2 ,\, \mu \in \Sigma_e \setminus \Delta_{fe} \}$ and a family of partial unitary representations $v_{e} \colon g \mapsto v_{e,g}$ of $G_e$ satisfying \ref{itm:D1}--\ref{itm:D4}. Since $B$ is a direct sum of full group $C^*$-algebras, universality yields a $*$-homomorphism $\pi \colon B \to C$ such that $\pi(\epsilon_g^e) = v_{e,g}$.

	Let  $fe,\,pq \in \Gamma^2$, $\mu \in \Sigma_e \setminus \Delta_{fe}$, $\nu \in \Sigma_q \setminus \Delta_{pq}$,  $g \in G_e$, and $h \in G_q$. It follows from \ref{itm:D1} and \ref{itm:D4} that $t_{f \mu e} t_{f \mu e}^*$ and $t_{p \nu q}t_{p \nu q}^*$ are mutually orthogonal projections unless $f\mu e = p \nu q$. Accordingly, \ref{itm:D3} and \eqref{eq:preip} imply that
	\begin{align}\label{eq:geninnerprod}
		\begin{split}
			(t_{f \mu e}u_{e,g})^* t_{p \nu q}u_{q,h}
			&= v_{e,g}^* t_{f \mu e}^* t_{p \nu q}v_{q,h}
			= v_{e,g}^* t_{f \mu e}^*t_{f \mu e}t_{f \mu e}^* t_{p \nu q} t_{p \nu q}^* t_{p \nu q}v_{q,h}\\
			&= \delta_{f \mu e, p \nu q} v_{e,g}^* t_{f \mu e}^* t_{f \mu e} v_{e,h}\\
			&= \delta_{f \mu e, p \nu q} v_{e,g^{-1}h}\\
			&=  \delta_{f \mu e, p \nu q} \pi(\epsilon^e_{g^{-1}h})\\
			&= \pi ((\chi_{\mu}^{fe} \cdot \epsilon^e_g \mid \chi_{\nu}^{pq} \cdot \epsilon^q_h)_{B}).
		\end{split}
	\end{align}
	In particular, $\|\xi_{\mu}^{fe} \cdot \epsilon^e_g \|	\le \|t_{f \mu e}v_{e,g}\|$.
	It follows that there is a norm-decreasing linear map $\psi \colon D \to C$ satisfying $\psi(\chi_{\mu}^{fe} \cdot \epsilon_g^e) = t_{f\mu e} v_{e,g}$ for all $fe \in \Gamma^2$, $\mu \in \Sigma_e \setminus \Delta_{fe}$ and $g \in G_e$. As $\psi(\chi_{\mu}^{fe} \cdot \epsilon_g^e) = \psi(\chi_\mu^{fe}) \pi(\epsilon_g^e)$ it  follows from a standard argument that $\psi(\xi \cdot a) = \psi(\xi) \pi(a)$ for all $\xi \in D$ and $a \in B$. Similarly, \eqref{eq:geninnerprod} implies that $\psi(\xi)^*\psi(\eta) = \pi((\xi\mid \eta)_B)$ for all $\xi,\, \eta \in D$.

	To see that the left action is respected, observe that \eqref{eq:leftactionD} together with \ref{itm:D2} imply that
	\begin{align*}
		\pi(\epsilon_g^q) \psi(\chi^{fe}_\mu) =  \delta_{q,f} v_{f,g} t_{f \mu e} = \delta_{q,f}  t_{f(\alpha_{\ol{f}}(g) \cdot \mu)e} v_{e,c_{fe}(g,\mu)}=
		\delta_{q,f} \psi(\chi_{\alpha_{\ol{f}}(g) \cdot \mu}^{fe} \cdot \epsilon_{c_{fe}(g,\mu)}^e) = \psi(\ol{\varphi}(\epsilon_g^q) \chi^{fe}_\mu ).
	\end{align*}
	Again, a standard argument shows that $\psi(\ol{\varphi}(a)\xi) = \pi(a)\psi(\xi)$ for all $\xi \in D$ and $a \in B$. Hence, $(\pi,\psi)$ is a representation of $(\ol{\varphi},D)$ in $C$.

	For Cuntz-Pimsner covariance, we use the fact that $(\chi_{\mu}^{fe})_{fe \in \Gamma^2, \mu \in \Sigma_e \setminus \Delta_{fe}}$ constitutes a frame for $D$. So $\ol{\varphi}(a) = \sum_{fe \in \Gamma^2} \sum_{\mu \in \Sigma_e \setminus \Delta_{fe}} \Theta_{\ol{\varphi}(a)\chi_{\mu}^{fe},\chi_{\mu}^{fe}}$. It then follows from \ref{itm:D4} that for all $f \in e$ and $g \in G_e$,
	\begin{align*}
		\psi^{(1)} \circ \ol{\varphi}(\epsilon_g^e)
		 & = \psi^{(1)} \Big(  \sum_{\substack{r(e) = s(f) \\ \mu \in \Sigma_{e} \setminus \Delta_{fe}}}  \Theta_{\ol{\varphi}(\epsilon_g^e)\chi_{\mu}^{fe},\chi_{\mu}^{fe}} \Big)
		= \sum_{\substack{r(e) = s(f)                      \\ \mu \in \Sigma_{e} \setminus \Delta_{fe}}} \pi(\epsilon_g^e)\psi(\chi_{\mu}^{fe}) \psi(\chi_{\mu}^{fe})^*\\
		 & = \sum_{\substack{r(e) = s(f)                   \\ \mu \in \Sigma_{e} \setminus \Delta_{fe}}} v_{e,g} t_{f \mu e} t_{f \mu e}^*
		= v_{e,g}
		= \pi(\epsilon_g^e).
	\end{align*}
	It once again follows from a standard argument that $\psi^{(1)} \circ \ol{\varphi} (a) = \pi(a)$ for all $a \in B$, so $(\ol{\varphi},D)$ is Cuntz-Pimsner covariant.

	The universal property of $\Oo_D$ as a Cuntz-Pimsner algebra yields a unique $*$-homomorphism $\Phi \colon \Oo_D \to C$ such that $\Phi \circ i_D = \psi$ and $\Phi \circ i_B = \pi$. Moreover, $\Phi$ is the unique $*$-homomorphism satisfying $\Phi(s_{f \mu e}) = t_{f \mu e}$ and $\Phi(u_{e,g}) = v_{e,g}$ for all $fe \in \Gamma^2$, $\mu \in \Sigma_e \setminus \Delta_{fe}$, and $g \in G_e$. Hence, $\Oo_D$ is universal for the relations \ref{itm:D1}--\ref{itm:D4}.
\end{proof}

\begin{rmk}
	Similar considerations to the proof of Proposition~\ref{prop:Dgenrel} can also provide a set of generators and relations for $\Oo_E\cong C^*(\Gg)$ which differ from the usual $\Gg$-family generators. We omit this description as it is more complicated than the corresponding one for $\Oo_D$, and we do not make use of it.
\end{rmk}

Exel and Pardo  \cite{EP17} introduced a class of $C^*$-algebras, now referred to as \hl{Exel-Pardo algebras}, which simultaneously generalise Nekrashevych's algebras \cite{Nek04,Nek09} for self-similar actions, as well as class of $C^*$-algebras introduced by Katsura \cite{Kat08AB} which describe all UCT Kirchberg algebras.

\begin{dfn}[\cite{EP17}]\label{dfn:exelpardo}
	Let $G$ be a countable discrete group acting on a finite directed graph $\Lambda = (\Lambda^0,\Lambda^1,r,s)$, and suppose that $c \colon G \times \Lambda^1 \to G$ is a cocycle for the action satisfying $c(g,e) \cdot x = g \cdot x$ for all $g \in G$, $e \in \Lambda^1$, and $x \in \Lambda^0$.  The \hl{Exel-Pardo algebra} $\Oo_{\Lambda,G}$ is the universal unital $C^*$-algebra generated by
	\[
		\{	p_x \mid x \in \Lambda^0 \} \cup \{v_e \colon e \in \Lambda^1\} \cup \{w_g \colon g \in G\}
	\]
	subject to relations:
	\begin{enumerate}[label={ (EP\arabic*)},align=left]
		\setlength\itemsep{0.5em}
		\item \label{itm:EP1}$\{	p_x \mid x \in \Lambda^0 \} \cup \{v_e \colon e \in \Lambda^1\}$ is a Cuntz-Krieger $\Lambda$-family in the sense of \cite{Rae05};
		\item \label{itm:EP2}$g \mapsto w_g$ is a unitary representation of $G$;
		\item \label{itm:EP3}$w_g v_e = v_{g\cdot e} w_{c(g,e)}$ for all $g \in G$ and $e \in \Lambda^1$; and
		\item \label{itm:EP4}$w_gp_x = p_{g \cdot x} w_g$ for all $g \in G$ and $x \in \Lambda^0$.
	\end{enumerate}
\end{dfn}

We claim that for a graph of groups with finite underlying graph, the algebra $\Oo_D$ is an Exel-Pardo algebra and give an explicit construction. To this end, let $\Gg=(\Gamma,G)$ be a locally finite nonsingular graph of countable discrete groups with a fixed choice of transversals $\Sigma_e$ for each $e \in \Gamma^1$. Moreover, suppose that $\Gamma^1$ is actually finite.

Consider the directed graph
$
	\Lambda_{\Gg} = (\Lambda_{\Gg}^0,\Lambda_{\Gg}^1,r_{\Lambda},s_{\Lambda})
$
with vertices $\Lambda_\Gg^0 = \Gamma^1$, edges $\Lambda_\Gg^1 = \{f\mu e \mid fe \in \Gamma^2,\, \mu \in \Sigma_e \setminus \Delta_{fe}\}$, range map $r_\Lambda(f\mu e) = f$, and source map  $s_\Lambda(f\mu e) = e$. Local finiteness of $\Gg$ implies that $\Lambda$ is locally finite, and nonsingularity of $\Gg$ implies that $\Lambda$ has no sources. Note that if $\Sigma_{e}= \{1\}$, then there are no edges from the vertex $e$ to $\ol{e}$.

Define an action of $g = (g_k)_{k \in \Gamma^1} \in \prod_{k \in \Gamma^1} G_k $ on $\Lambda_\Gg$ by
\begin{equation}\label{eq:epaction}
	g \cdot e = e \quad \text{and} \quad g \cdot f \mu e = 
		f (\alpha_{\ol{f}}(g_f) \cdot \mu) e 
\end{equation}
for all $e \in \Lambda_\Gg^0$ and $f \mu e \in \Lambda_\Gg^1$.
The structure of the directed graph $\Lambda_\Gg$ and the action of $\prod G_e$ on $\Lambda_\Gg$ are independent of the choice of transversals.
For each $f \in \Gamma^1$ let $\iota_f \colon G_f \to \prod G_e$ denote the natural inclusion and define a map $c \colon \prod G_e \times \Lambda_\Gg^1 \to \Lambda_\Gg^1$ by
\begin{equation}\label{eq:epcocycle}
	c(g,f\mu e) = \iota_e \circ c_{fe}(g_f,\mu).
\end{equation}
Since each $c_{fe}$ is a cocycle it follows that $c$ is a cocycle for the action of $\prod G_e$ on $\Lambda_\Gg^1$. Moreover, since the action of $\prod G_e$ fixes $\Lambda_\Gg^0$, the cocycle trivially satisfies $c(g,f\mu e) \cdot e = g \cdot e$ for all $e \in \Lambda_\Gg^0$.

\begin{thm}\label{thm:exelpardo}
	Let $\Gg=(\Gamma,G)$ be a locally finite nonsingular graph of countable discrete groups with $\Gamma$ a finite graph. Then
	$\Oo_D$ is isomorphic to the Exel-Pardo algebra $\Oo_{\Lambda_{\Gg}, \prod_{e \in \Gamma^1} G_e}$ with cocycle $c$ given by \eqref{eq:epcocycle}.
\end{thm}

\begin{proof}
	We show that the defining relations for $\Oo_D$ from Proposition~\ref{prop:Dgenrel} agree with the defining relations of the Exel-Pardo algebra  $\Oo_{\Lambda_{\Gg}, \prod_{e \in \Gamma^1} G_e}$ with cocycle $c$. Suppose that $\{u_{e,g} \colon e \in \Gamma^1,\, g \in G_e\} \cup \{s_{f\mu e} \colon f\mu e \in \Lambda_{\Gg}^1\}$ satisfy the hypotheses of Proposition~\ref{prop:Dgenrel}, and that $\{	p_x \mid x \in \Lambda_\Gg^0 \} \cup \{v_e \colon e \in \Lambda_\Gg^1\} \cup \{w_g \colon g \in G\}$ satisfy the hypotheses of Definition~\ref{dfn:exelpardo}.

	Clearly \ref{itm:D1}, \ref{itm:D3}, and \ref{itm:D4} are equivalent to $\{u_{e,1} \colon e \in \Gamma^1\} \cup \{s_{f\mu e} \colon f\mu e \in \Lambda_{\Gg}^1\}$ being a Cuntz-Krieger $\Lambda_\Gg$-family. Moreover, we have the following correspondence between the two families of elements:
	\begin{align*}
		p_e                    & \xleftrightarrow{\quad} u_{e,1}                         \\
		v_{f \mu e}            & \xleftrightarrow{\quad} s_{f \mu e}                     \\
		w_g                    & \xleftrightarrow{\quad} \sum_{e \in \Gamma^1} u_{e,g_e} \\
		p_e w_{\iota_e(h)} p_e & \xleftrightarrow{\quad} u_{e,h}	.
	\end{align*}
	Since $\Gamma^1$ is finite, $\sum_{e \in \Gamma^1} u_{e,1} = 1$. So $g \mapsto \sum_{e \in \Gamma^1} u_{e,g_e}$ is a unitary representation of $\prod G_e$, and $h \mapsto p_e w_{\iota_e(h)} p_e$ is a partial unitary representation of $G_e$.
	The condition \ref{itm:D2} corresponds directly to the condition \ref{itm:EP3}, and \ref{itm:EP4} is trivally satisfied since the action of $\prod G_e$ on $\Lambda_\Gg$ fixes vertices.
\end{proof}

\begin{example}
	If $\Gg = (\Gamma,G)$ is a graph of groups with trival edge groups, then $\prod G_e$ is trivial and $\Oo_{\Lambda_{\Gg},\{1\}}$ is just the graph $C^*$-algebra $C^*(\Lambda_\Gg)$. Note that $\Lambda_\Gg$ is typically not the directed graph $E_\Gg$ of \cite[Theorem 3.6]{BMPST17}.
\end{example}

\begin{example}
	Let $\Gg = (\Gamma,G)$ be the graph of groups associated to $BS(1,2)$ from Example~\ref{ex:DinfBS}. Then $\Lambda_\Gg$ is the directed graph defined by the diagram
	
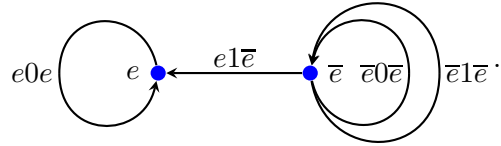
\[
		\begin{tikzpicture}[scale=2,baseline=0.5ex, vertex/.style={circle, fill=black, inner sep=2pt}]
			\begin{scope}
				clip (-2,-.6) rectangle (2,.6);

				\node[vertex,blue,label={[label distance=0.05]right:{$ \ol{e}$}}] (eb) at (0,0) [circle] {};
				\node[vertex,blue,label={[label distance=0.05]left:{$ e$}}] (e) at (-1,0) [circle] {};


				\draw[-stealth,thick] (eb) -- (e) node[pos=0.5, inner sep=2pt, anchor=south]{${e 1 \ol{e}}$};

				\draw[-stealth,thick] (.85,0) .. controls (.85,.6) and (.1,.6) .. (eb) node[pos=0, inner sep=2pt, anchor=west] {${\ol{e} 1\ol{e}}$};

				\draw[thick] (.85,0) .. controls (.85,-.6) and (.1,-.6) .. (eb);

				\draw[-stealth,thick] (.65,0) .. controls (.65,.45) and (.1,.45) .. (eb) node[pos=0, inner sep=2pt, anchor=east] {${\ol{e} 0\ol{e}}$};

				\draw[thick] (.65,0) .. controls (.65,-.45) and (.1,-.45) .. (eb);

				\draw[thick] (e) .. controls (-1.1,.45) and (-1.65,.45) .. (-1.65,0);

				\draw[-stealth,thick] (-1.65,0) .. controls (-1.65,-.45) and (-1.1,-.45) .. (e)  node[pos=0, inner sep=2pt, anchor=east] {${e 0e}$};

			\end{scope}

		\end{tikzpicture}.
	\]
	The action of $(a,b) \in G_e \times G_{\ol{e}} \cong \ZZ^2$ on $\Lambda_\Gg$ is such that $(a,0)$ acts trivially, while $(0,b)$ fixes both $e 0 e$ and $e 1 \ol{e}$, and
	\[
		(0,b) \cdot \ol{e} k \ol{e} = \ol{e} ( (b+k)\bmod 2) \ol{e}.
	\]
	The cocycle $c$ is given by
	\[
		c((a,b),x) = \begin{cases}
			\big(0,\frac{b+k -( (b+k)\bmod 2)}{2}\big) & \text{if } x = \ol{e} k \ol{e}, \\
			(0,a)                                      & \text{if } x = e 0 \ol{e},      \\
			(2a,0)                                     & \text{if } x = e 0 e.
		\end{cases}
	\]
	For this example, the triple $(\ZZ^2, \Lambda_\Gg, c)$  shares similarities with both \cite[Example 3.4]{EP17} and \cite[Example 3.3]{EP17}, but is not an instance of either example.
\end{example}

We finish this section with a summary of results.  Recall that the Bass-Serre Theorem \cite[Theorem 13]{Ser80} implies that if $G$ is a group acting on a tree $X$, then $G$ is isomorphic to the fundamental group of the associated quotient graph of groups $\Gg$, and $X$ is $G$-equivariantly isomorphic to the universal covering tree of $\Gg$. The following is now a synthesis of Theorem~\ref{thm:exelpardo}, Proposition~\ref{prop:DandEmoritaeq}, Theorem~\ref{thm:isomorphism}, and Theorem~\ref{thm:bmpst17iso}.
\begin{cor}
	\label{cor:summary}
	Suppose that $G$ is a countable group acting without inversions on a locally finite nonsingular tree $X$, and let $\tau$ denote the induced action on $C(\partial X)$. Let $\Gg = (\Gamma,G)$ denote the quotient graph of groups associated to the action of $G$ on $X$. Then with $(\ol{\varphi},D)$ the graph of groups correspondence for $\Gg$ and $({\varphi},E)$ as the amplified graph of groups correspondence,
	\[
		\Oo_D \sim_{me}	\Oo_E\cong C^*(\Gg) \sim_{me} C(\partial X_{\Gg}) \rtimes_{\tau} \pi_1(\Gg) \cong C(\partial X) \rtimes_{\tau} G ,
	\]
	where $\sim_{me}$ denotes Morita equivalence.
	In addition, if $\Gamma^1 = G \bs X^1$ is finite, then $\Oo_{\Lambda_{\Gg}, \prod_{e \in Y^1} G_e}$ is isomorphic to $\Oo_D$.

\end{cor}

\section{$K$-theory}
\label{sec:Ktheory}

The $K$-theory of Cuntz-Pimsner algebras is well-understood. In his seminal paper, Pimsner \cite{Pim97} showed that for a large class of $C^*$-correspondences, the $K$-theory of the associated Cuntz-Pimsner algebra could be computed using the Kasparov class of the correspondence. Although we make use of $KK$-theory in this section, the Kasparov products we consider are relatively tame, and we will not need $\ZZ_2$-graded $C^*$-algebras.

\begin{dfn}
	An  \hl{odd Kasparov $A$--$B$-module} $(A,{}_{\phi}E_B,V)$ consists of a
	countably generated ungraded right Hilbert $B$-module $E$, with
	$\phi\colon A\to \End_B(E)$ a $*$-homomorphism, together with  $V\in \End_B(E)$, such that
	$a(V-V^*),\ a(V^2-1),\ [V,a]$ are all compact endomorphisms on $E$
	for
	all $a\in A$.

	An \hl{even Kasparov $A$--$B$-module} has, in addition,
	a grading by a self-adjoint
	endomorphism
	$\Gamma$ satisfying $\Gamma^2=1$,
	$\phi(a)\Gamma=\Gamma\phi(a)$ for all $a \in A$, and $V\Gamma+\Gamma V=0$.
\end{dfn}

We will not describe the equivalence relation on (even)
Kasparov $A$--$B$-modules which yields
$KK(A,B)$, referring the reader to \cite{Kas80}. We note that
correspondences with compact left action define even $KK$-classes, by taking the operator $V$ to be 0.
The Kasparov product of such correspondences is just the balanced tensor product.

For each $fe \in \Gamma^2$ let $[E_{fe}] := [(A_f,E_{fe},0)] \in KK(A_f,A_e)$ and $[D_{fe}] := [(B_f,D_{fe},0)] \in KK(B_f,B_e)$ denote the elements induced by the correspondences $E_{fe}$ and $D_{fe}$. Recall that $A_e$ and $B_e$ are Morita equivalent via the bimodule $F_e$ and $E_{fe} = F_f \otimes_{B_f} D_{fe} \otimes_{B_e} F_e^*$. Moreover, since $E$ and $D$ are direct sums, Proposition \ref{prop:DandEmoritaeq} implies that $A$ and $B$ are also Morita equivalent. Conjugating with $F=\bigoplus_{e \in \Gamma^1}F_e$ gives an isomorphism
$KK(B,B)\to KK(A,A)$ which carries $[D]$ to $[E]$.

For $i = 0,1$ the Kasparov product $\cdot \otimes_{B} [D] \colon K_i(B) \to K_i(B)$ can be thought of as an ``adjacency matrix'' of group homomorphisms indexed by $\Gamma^1 \times \Gamma^1$: the entry corresponding to $(f,e) \in \Gamma^1 \times \Gamma^1$ is the group homomorphism $\cdot \otimes_{A_f} [D_{fe}] \colon K_i(C^*(G_f)) \to K_i(C^*(G_e))$ if $s(f) = r(e)$, and the zero map otherwise.

Recall that Proposition~\ref{prop:DandEmoritaeq} together with Theorem~\ref{thm:isomorphism} imply that $\Oo_D$ is Morita equivalent to $C^*(\Gg)$. We have the following six-term sequence in $K$-theory which is analogous to the sequence for directed graph $C^*$-algebras (cf. \cite[Theorem 7.16]{Rae05}).

\begin{thm} \label{thm:ktheory} Let $\Gg = (\Gamma,G)$ be a locally finite nonsingular graph of countable groups. For $i = 0,1$ let $\Lambda_i = \cdot \ox_A (\id - [D]) \colon K_i(B) \to K_i(B)$. Let $m \colon K_*(\Oo_D) \to K_*(C^*(\Gg))$ denote the isomorphism induced by Morita equivalence of Corollary~\ref{cor:summary} and let $i_B: B\hookrightarrow \Oo_D$ denote the universal inclusion. Then the following six-term sequence of abelian groups is exact:
	\begin{center}
		\begin{tikzpicture}[baseline=-0.5ex]
			\matrix (m) [matrix of math nodes,row sep=2.5em,column sep=4.em,minimum width=2em]
			{
				\bigoplus_{e \in \Gamma^1} K_0(C^*(G_e)) & \bigoplus_{e \in \Gamma^1} K_0(C^*(G_e)) & K_0(C^*(\Gg))                         \\
				K_1(C^*(\Gg))                         & \bigoplus_{e \in \Gamma^1} K_1(C^*(G_e)) & \bigoplus_{e \in \Gamma^1} K_1(C^*(G_e)) \\
			};
			\path[-stealth]
			(m-1-1) edge node [above] {$\Lambda_0$} (m-1-2)
			(m-1-2) edge node [above] {$m \circ (i_{B})_* $} (m-1-3)
			(m-1-3) edge node [right] {$\partial$} (m-2-3)
			(m-2-3) edge node [above] {$\Lambda_1$} (m-2-2)
			(m-2-2) edge node [above] {$m \circ (i_{B})_* $} (m-2-1)
			(m-2-1) edge node [left] {$\partial$} (m-1-1)
			;
		\end{tikzpicture}.
	\end{center}
\end{thm}
\begin{proof}
	This follows immediately from \cite[Theorem 4.9]{Pim97}.
\end{proof}
\begin{rmk}
	The corresponding six-term sequences in $KK$-theory (cf. \cite[Theorem 4.9]{Pim97}) also hold with the caveat of nuclearity being inserted where appropriate.
\end{rmk}

Corollary~\ref{cor:summary} implies that Theorem~\ref{thm:ktheory} can be reinterpreted in terms of group actions on the boundary of a tree.

\begin{cor} \label{cor:grouptree}
	Suppose that $G$ is a countable group acting without inversions on a locally finite nonsingular tree $X$ and let $\tau$ denote the induced action on $C(\partial X)$. Suppose that $Y = (Y^0,Y^1,r_Y,s_Y)$ is a fundamental domain for the action of $G$ on $X$. For each $e \in Y^1$ let $G_e = \stab_G(e)$. Let $m \colon K_*(\Oo_D) \to K_*(C(\partial X) \rtimes_\tau G)$ denote the isomorphism induced by Morita equivalence  and let $\iota_B \colon B \to \Oo_D$ denote the universal inclusion.  Then the following six-term sequence of abelian groups is exact:
	\begin{center}
		\begin{tikzpicture}[baseline=-0.5ex]
			\matrix (m) [matrix of math nodes,row sep=2.5em,column sep=4.em,minimum width=2em]
			{
				\bigoplus_{e \in Y^1} K_0(C^*(G_e)) & \bigoplus_{e \in Y^1} K_0(C^*(G_e)) & K_0(C(\partial X) \rtimes_\tau G)     \\
				K_1(C(\partial X) \rtimes_\tau G)     & \bigoplus_{e \in Y^1} K_1(C^*(G_e)) & \bigoplus_{e \in Y^1} K_1(C^*(G_e)) \\
			};
			\path[-stealth]
			(m-1-1) edge node [above] {$\Lambda_0$} (m-1-2)
			(m-1-2) edge node [above] {$m \circ (i_{B})_* $} (m-1-3)
			(m-1-3) edge node [right] {$\partial$} (m-2-3)
			(m-2-3) edge node [above] {$\Lambda_1$} (m-2-2)
			(m-2-2) edge node [above] {$m \circ (i_{B})_* $} (m-2-1)
			(m-2-1) edge node [left] {$\partial$} (m-1-1)
			;
		\end{tikzpicture}.
	\end{center}
\end{cor}
\begin{proof}
	Let $\Gg$ denote the quotient graph of groups associated to the action of $G$ on $X$ as outlined in Section~\ref{sec:prelims}. In particular, $\Gamma \simeq G \bs X$ and for each $e \in \Gamma^1$ the group $G_e$ is the stabiliser of the unique lifted edge $e \in Y^1$. Since $\Oo_D$ is Morita equivalent to $C(\partial X) \rtimes_{\tau} G$, Theorem~\ref{thm:ktheory} yields the result.
\end{proof}

\begin{rmk}
	Corollary~\ref{cor:grouptree} bears a resemblance to a special case of \cite[Theorem 16]{Pim86}. However, here only the edge stabilizers are explicitly used to compute the $K$-theory of $C(\partial X) \rtimes G$. The data concerning the vertex stabilizers is encoded within the module $E$. It is not immediately clear to the authors how to translate between these two pictures.
\end{rmk}

We now move to compute the $K$-theory of some examples, especially when the vertex groups are abelian.
Although the maps $\Lambda_i$ involve taking a Kasparov product, the actual computations are relatively tame in simple examples. In what follows, we let $1_e$ denote the identity in $C^*(G_e)$ for all $e \in \Gamma^1$.

\begin{lem}\label{lem:kunit}
	Let $\Gg = (\Gamma,G)$ be a locally finite nonsingular graph of countable groups.
	For each $fe \in \Gamma^2$,
	\[
		[1_{f}] \otimes_B [D_{fe}] = (|\Sigma_e|-\delta_{f,\ol{e}}) [1_{e}],
	\]
	where $[1_{f}] \in K_0(C^*(G_f)) \cong KK(\CC, C^*(G_f))$.
\end{lem}
\begin{proof}
	Since the left action of $C^*(G_f)$ on $D_{fe}$ is unital, it follows that
	\begin{equation*}
		[1_f] \otimes_B [D_{fe}] = \Big[\Big(\CC, \bigoplus_{\mu \in \Sigma_e \setminus \Delta_{fe}}C^*(G_e),0\Big)\Big] = (|\Sigma_e|-\delta_{f,\ol{e}}) [1_e].
		\qedhere
	\end{equation*}
\end{proof}

\begin{lem}\label{lem:eebarkk}
	Let $\Gg = (\Gamma,G)$ be a locally finite nonsingular graph of countable groups. Suppose that $v \in \Gamma^0$ is such that $G_v$ is abelian. Then for all $e \in \Gamma^1$ with $r(e) = v$ we have
	\[
		[D_{\ol{e}e}] = (|\Sigma_e| - 1)\id_{KK(C^*(G_{\ol{e}}),C^*(G_e))}\]
	in $KK(C^*(G_{\ol{e}}),C^*(G_e))$, where we identify $C^*(G_{\ol{e}})=C^*(G_e)$.
\end{lem}
\begin{proof}
	Since $G_v$ is abelian, the left and right cosets of $\alpha_e(G_e)$ in $G_v$ are equal. 
	The result then follows immediately from \eqref{eq:leftactionD}, since the left and right actions of $C^*(G_e)$ on $D_{\ol{e}e} = \bigoplus_{\mu \in \Sigma_e \setminus \{1\} }C^*(G_e)$ agree after identifying $G_e$ with $G_{\ol{e}}$. Consequently,
	\[
		[(C^*(G_e), D_{\ol{e}e} , 0)] = \sum_{i=1}^{|\Sigma_e|-1} \id_{KK(C^*(G_{\ol{e}}),C^*(G_e))}.
	\]
	Note that in the case where $\alpha_e:\,G_e\to G_{r(e)}$ is surjective we have $[D_{\ol{e}e}] = 0$.
\end{proof}

\subsection{Edges of groups}
We can now say something about the $K$-theory of the $C^*$-algebra of an edge of groups,
that is a graph of groups consisting of a single edge.
Consider the edge of groups,
\vspace*{10pt}
\[
	\Gg =
	\begin{tikzpicture}
		[baseline=0.25ex,vertex/.style={circle, fill=black, inner sep=2pt}]

		\node[vertex,blue] (a) at (0,0) {};%
		\node[vertex,red] (b) at (3,0) {};%
		\node[inner sep = 2pt, anchor=south] at (a.north) {$G_v$};%
		\node[inner sep = 2pt, anchor=south] at (b.north) {$G_w$};%
		\draw[-latex,thick] (b) -- (a)
		node[pos=0.5, anchor=south,inner sep=2.5pt] {$G_e$}
		node[pos=0.25, anchor=south,inner sep=2pt] {\scriptsize$\alpha_{\ol{e}}$}
		node[pos=0.75, anchor=south,inner sep=2pt] {\scriptsize$\alpha_{e}$} ;
	\end{tikzpicture}.
	\vspace*{10pt}
\]
Assuming that $\Gg$ is non-singular it follows that $|\Sigma_e| \ge 2$ and $|\Sigma_{\ol{e}}| \ge 2$. Then $K_i(B) \cong K_i(C^*(G_e)) \oplus K_i(C^*(G_{\ol{e}}))$. The only paths of length two in $\Gg$ are $e \ol{e}$ and $\ol{e}e$ so $[D] = [D_{\ol{e}e}] \oplus [D_{e \ol{e}}]$.

If we assume that $G_v$ and $G_w$ are abelian, then
Lemma~\ref{lem:eebarkk} implies that for a class $(x,y) \in   K_i(B)$ we have
\begin{equation}
	\label{eq:multieeb}
	(x,y)\otimes_B (\id_* - [D] ) =\big (x - (|\Sigma_e|-1)y, y - (|\Sigma_{\ol{e}}|-1)x\big).
\end{equation}
We consider a few specific cases of edges of groups.
\begin{example}\label{ex:Ktheoryeog1}
	Fix $m,n \ge 2$ and let $\Gg = (G,\Gamma)$ be an edge of groups with $G_v = \ZZ/n\ZZ$, $G_w = \ZZ/m\ZZ$, $G_e = \{0\}$, and monomorphisms $\alpha_e$ and $\alpha_{\ol{e}}$ being the inclusion of $\{0\}$. Then $|\Sigma_e| = n$ and $|\Sigma_{\ol{e}}| = m$.
	In this case we have $C^*(G_e) = C^*(G_{\ol{e}}) \cong \CC$ so
	\[
		K_0(B) = \ZZ[1_e] \oplus \ZZ[1_{\ol{e}}] \quad \text{and} \quad K_1(B) = 0.
	\]

	Applying the six-term sequence of Theorem~\ref{thm:ktheory}, it follows that $K_0(C^*(\Gg)) \cong \coker(\Lambda_0)$ and $K_1(C^*(\Gg)) \cong \ker(\Lambda_0)$ where $\Lambda_0 \colon \ZZ[1_e] \oplus \ZZ[1_{\ol{e}}] \to \ZZ[1_e] \oplus \ZZ[1_{\ol{e}}]$ is the $\ZZ$-linear map given by \eqref{eq:multieeb}.
	Treating $[1_e]$ as the column $(\begin{smallmatrix}
			1 & 0
		\end{smallmatrix})^T$ and $[1_{\ol{e}}]$ as the column $(\begin{smallmatrix}
			0 & 1
		\end{smallmatrix})^T$
	the map $\Lambda_0$ has matrix representation
	\[
		\begin{pmatrix}
			1 & n-1 \\ m-1 & 1
		\end{pmatrix}
		\text{ with Smith normal form }
		\begin{pmatrix}
			1 & 0 \\ 0 & 1 - (n-1)(m-1)
		\end{pmatrix}.
	\]
	It now follows that
	\[
		K_0(C^*(\Gg)) \cong \frac{\ZZ}{{(1-(n-1)(m-1))}\ZZ}
		\quad \text{and} \quad
		K_1(C^*(\Gg)) \cong \begin{cases}
			\ZZ & \text{if } m=n=2; \\
			0   & \text{otherwise.}
		\end{cases}
	\]

	Since the edge groups are trivial in this example, \cite[Theorem 3.6]{BMPST17} shows that $C^*(\Gg)$ is isomorphic to the $C^*$-algebra of a directed graph $E_\Gg$. One can readily check that the $K$-theory computed above agrees with that of the directed graph $C^*$-algebra $C^*(E_\Gg)$. For the specific case of the action of $PSL_2(\ZZ) \cong \ZZ_2 * \ZZ_3$ on the boundary $\partial X$ of the tree $X$ considered in Example~\ref{ex:PSL2Z} we have $K_i(C(\partial X) \rtimes PSL_2(\ZZ)  ) = 0$ for $i =0,1$.
\end{example}

\begin{example}\label{ex:Ktheoryeog2}
	Fix $m,n \in \ZZ$ such that $|m|,|n| \ge 2$ and let $\Gg = (G,\Gamma)$ be an edge of groups with $G_v = G_w = G_e = \ZZ$. Let $\alpha_e$ denote multiplication by $n$ and let $\alpha_{\ol{e}}$ denote multiplication by $m$. Then $|\Sigma_e| = |n|$, $|\Sigma_{\ol{e}}| = |m|$, and $C^*(G_e) = C^*(G_{\ol{e}}) \cong C(\TT)$. Letting $u_e$ denote the unitary $z \mapsto z$ in $K_1(C^*(G_e))$, it follows that
	\[
		K_0(B) = \ZZ[1_e] \oplus \ZZ[1_{\ol{e}}] \quad
		\text{and} \quad K_1(B) = \ZZ[u_e] \oplus \ZZ[u_{\ol{e}}].
	\]
	For $i=0,1$ the groups $\ker(\Lambda_i)$ are free abelian since they are subgroups of the free abelian group $K_i(B)$.   It now follows from Theorem~\ref{thm:ktheory} that
	$
		K_i(C^*(\Gg)) \cong \ker(\Lambda_{1-i}) \oplus \coker(\Lambda_i).
	$
	A similar argument to Example~\ref{ex:Ktheoryeog1} now shows that for $i = 0,1$,
	\[
		K_i(C^*(\Gg)) \cong \begin{cases}
			\ZZ                                 & \text{if } |n|=|m|=2; \\
			\frac{\ZZ}{{(1-(|n|-1)(|m|-1))}\ZZ} & \text{otherwise.}
		\end{cases}
	\]
\end{example}

\subsection{Graphs of trivial groups}\label{subsec:trivialgog}
Suppose that $\Gg = (\Gamma, G)$ is a locally finite nonsingular graph of groups for which each edge group and vertex group is trivial. If $\Gamma$ is also a finite graph, then the algebra $C^*(\Gg)$ is the Cuntz-Krieger algebra associated to $\Gamma$ by Cornelissen, Lorscheid, and Marcolli in \cite{CLM08} (see \cite[Remark 3.10]{BMPST17}).
In this case $C^*(G_e) \cong \CC$ for all $e \in \Gamma^1$, and for all $fe \in \Gamma^2$,
\[
	D_{fe} \cong \begin{cases}
		\CC & \text{if } f \ne \ol{e}; \\
		0   & \text{if } f = \ol{e}.
	\end{cases}
\]
In particular, $K_0(B) = \bigoplus_{e \in \Gamma^1}  \ZZ[1_e]$.
Since $K_1(B) = 0$, it follows from the six-term sequence of Theorem \ref{thm:ktheory} that
\[
	K_0(C^*(\Gg)) \cong \coker(\Lambda_0) \quad \text{and} \quad K_1(C^*(\Gg)) \cong \ker(\Lambda_1).
\]
On generating elements $[1_f] \in K_0(B)$ we have
\[
	[1_f]  \ox_B [D] = [1_f]  \ox_B \sum_{r(e) = s(f)} [D_{fe}] =  \Big(\sum_{r(e) = s(f)} [1_e]  \Big) - [1_{\ol{f}}] ,
\]
so that
\[
	\Lambda_0([1_f]) = [1_f] + [1_{\ol{f}}] - \sum_{r(e) = s(f)} [1_e].
\]
In particular, the Kasparov product $ \cdot \ox_B [D] \colon K_0(B) \to K_0(B)$ agrees with the operator $T$ of \cite[Section 2.2]{CLM08}. Accordingly, if $\Gamma$ is a finite graph then \cite[Theorem 1]{CLM08} can be used to compute $K_0(C^*(\Gg))$ and $K_1(C^*(\Gg))$ in terms of the number of the Betti number of $\Gamma$.

\subsection{Generalised Baumslag-Solitar graphs of groups}
A \hl{generalised Baumslag-Solitar (GBS) graph of groups} is a graph of groups where all of the edge and vertex groups are isomorphic to $\ZZ$. The fundamental group of such a graph of groups is called a \hl{generalised Baumslag-Solitar (GBS)  group}. A survey of results concerning GBS groups has been compiled by Robinson \cite{Rob15}.

For the remainder of this subsection, let $\Gg = (\Gamma,G)$ be a GBS graph of groups. For each  $e \in \Gamma^1$ we have $G_e \cong \ZZ$, and so $C^*(G_e) \cong C(\TT)$, and $K_i(C^*(G_e)) \cong \ZZ$ for $i=0,1$.

\begin{prop}\label{prop:genbs}
	Let $\Gg = (\Gamma,G)$ be a locally finite nonsingular GBS graph of groups. Let $fe \in \Gamma^2$ with $f \ne \ol{e}$ and suppose that $\alpha_e$ is given by multiplication by $n$, and $\alpha_{\ol{f}}$ is given by multiplication by $m$ for some $n,\,m \in \ZZ \setminus \{0\}$. Then the Kasparov product $\cdot \otimes_{A} [D_{fe}] \colon K_i(
		C^*(G_f)) \to K_i(C^*(G_e))$ acts as multiplication by $|n|$ on $K_0$ and multiplication by $\sgn(n)m$ on $K_1$, where $\sgn(x) = x/|x|$ is the sign function.
\end{prop}
\begin{proof}
	The $K_0$-statement follows from Lemma~\ref{lem:kunit} and the fact $K_0(C(\TT))$ is generated by the class of the unit.
	For $K_1$ fix a transversal $\Sigma_e =\{0,1,\ldots,|n|-1\}$.
	The action of $x \in G_{r(e)}$ on $\mu \in \Sigma_e$ is given by $x \cdot \mu = (x + \mu) \bmod |n|$. Using \eqref{eq:cocyclee} the corresponding cocycle is given by $c_{e}(x,\mu) = \frac{x + \mu - x \cdot \mu}{n}$: the result of integer division of $x + \mu$ by $n$. The induced action of $x \in G_f$ via $\ol{f}$ on $\mu \in \Sigma_e$ is given by $x \cdot \mu = (mx + \mu) \bmod |n|$.  The cocycle for the action of $G_f$ is given by  $c_{fe}(x,\mu) = \frac{mx+\mu - x \cdot \mu}{n}$.

	It follows from Equation~\eqref{eq:leftactionD} that the left action of the point mass $\epsilon_n^{f} \in C_c(\ZZ) \subset C^*(G_f)$ on $\xi \in C_c(\ZZ)^{|n|} \subset D_{fe}$ is given by
	\[
		\varphi_{fe}(\epsilon_n^{f}) \xi(\mu,k) = \xi\big((-n) \cdot \mu, c_{fe}(-n,\mu ) + k\big) = \xi(\mu,k-m) =( \xi \cdot \epsilon_m^{e})(\mu,k).
	\]
	Let $u_{f}$ and $u_{e}$ denote the unitary $z \mapsto z$ in $C^*(G_f)$ and $C^*(G_e)$, respectively. The preceding computation implies that $\varphi_{fe}(u_f^n)\xi = \xi \cdot u_e^m$ for all $\xi \in D_{fe}$.

	Recall that the group $K_1(C(\TT))\cong \ZZ$ is generated by
	the class of the unitary $z \mapsto z$. To facilitate
	the computation of Kasparov products we utilise a special case of
	\cite[Section 3]{BKR19}. By utilising the isomorphism $K_1(C(\TT)) \cong KK^1(\CC,C(\TT))$ we obtain  a new generator,  the class of the (unbounded) Kasparov module $(\CC, \Xi,\Cc(u))$, where $\Cc(u)$ denotes the Cayley transform of the unitary $u \colon z \mapsto z$. In particular, $\Cc(u)$ is the unbounded regular self-adjoint operator with domain $\dom(\Cc(u)) = (u-1)C(\TT)$ satisfying $\Cc(u)v = i\frac{u+1}{u-1}v$ for all $v \in \dom(\Cc(u))$; and $\Xi$ is the right $C(\TT)$-module given by the closure of $\dom(\Cc(u))$ in the right $C(\TT)$-module $C(\TT)$.

	For each $n \in G_f$ denote the closure of $(u_f^n-1)C(\TT)$ in $C(\TT)$ by $\Xi_{f,n}$. In $C^*(G_f) \ox_{C^*(G_f)} D_{fe}$ we have
	\[
		(u_f^n-1) \ox \xi = 1 \ox \varphi_{fe}(u_f^n- 1)\xi= (1 \ox \xi)\cdot (u_e^m-1)
	\]
	for all $\xi \in D_{fe}$.
	Consequently, $\Xi_{f,n} \ox_{C^*(G_f)} D_{fe}$ is isomorphic to  $\Xi_{e,m}^{|n|}$. Moreover, it follows that in $K_1(C^*(G_e))$,
	\begin{align*}
		n[u_f] \otimes_A [D_{fe}] & =[u_f^n] \otimes_A [D_{fe}]                                              \\
		                          & = [(\CC, \Xi_{f,n}, \Cc(u_f^n)] \otimes_{C^*(G_f)} [(C^*(G_f),D_{fe},0)] \\
		                          & = [(\CC,\Xi_{f,n} \otimes_{C^*(G_f)} D_{fe},\Cc(u_f^n) \otimes 1)]       \\
		                          & = [(\CC, \Xi_{e,m}^{|n|}, \Cc(\diag(u_e^m)) ]                            \\
		                          & = |n| [u_e^m]                                                            \\
		                          & = |n|m [u_e].
	\end{align*}
	Dividing through by $n$ gives $[u_f] \ox_A [D_{fe}] = \sgn(n)m [u_e]$.
\end{proof}

\begin{example}
	\label{ex:loop}
	Fix $m,n \in \ZZ \setminus\{0\}$ and consider the loop of groups $\Gg$ from Example~\ref{ex:DinfBS} given by,
	
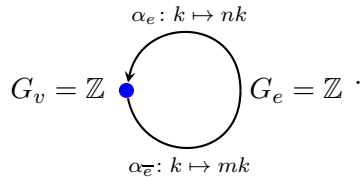
\[
		\begin{tikzpicture}[scale=2,baseline=0.5ex, vertex/.style={circle, fill=black, inner sep=2pt}]
			\begin{scope}
				\clip (-1,-.6) rectangle (1.5,.6);

				\node[vertex,blue] (0_0) at (0,0) [circle] {};


				\draw[-stealth,thick] (.75,0) .. controls (.75,.5) and (.1,.5) .. (0_0) node[pos=0, inner sep=3pt, anchor=west] {$G_e = \ZZ$}
				node[pos=0.5, anchor=south,inner sep=3pt]{\scriptsize $ \alpha_{{e}} \colon k \mapsto nk$};

				\draw[thick] (.75,0) .. controls (.75,-.5) and (.1,-.5) .. (0_0) node[pos=0.5, anchor=north,inner sep=3pt]{\scriptsize $ \alpha_{\ol{e}} \colon k \mapsto mk$};

				\draw (-.45,0) node {$G_v = \ZZ$};

			\end{scope}

		\end{tikzpicture}.
	\]
	Then $\Gg$ is the quotient graph of groups for an action of the Baumslag-Solitar group $BS(n,m) = \langle a,t \mid ta^nt^{-1} = a^m \rangle$ acting on a regular $(|m|+|n|)$-valent tree. The case where $n=2$ and $m=1$ was considered in detail in Example \ref{ex:DinfBS}.

	Observe that $B = C^*(G_e) \oplus C^*(G_{\ol{e}}) \cong C(\TT)^2$, so
	\[
		K_0(B) \cong \ZZ[1_e] \oplus \ZZ[1_{\ol{e}}] \quad \text{and} \quad
		K_1(B) \cong \ZZ[u_e] \oplus \ZZ[u_{\ol{e}}].
	\]

	Treat $[1_e]$ as the column $(\begin{smallmatrix}
			1 & 0
		\end{smallmatrix})^T$ and $[1_{\ol{e}}]$ as the column $(\begin{smallmatrix}
			0 & 1
		\end{smallmatrix})^T$. Then according to both Proposition~\ref{prop:genbs} and Lemma~\ref{lem:eebarkk} the map $\Lambda_0: K_0(A) \to K_0(A)$ of Theorem~\ref{thm:ktheory} can be identified with left multiplication by the matrix
	\[
		\begin{pmatrix}
			1-|n| &
			|n|-1   \\ |m|-1 & 1-|m|
		\end{pmatrix}
		\text{ which has Smith normal form }
		\begin{pmatrix}
			\gcd(1-|n|,1-|m|) & 0 \\ 0 & 0
		\end{pmatrix}
	\]
	by \cite[Theorem 2.4]{Sta16}.
	Here we recall that $\gcd(0,a) = a$ for $a \in \ZZ$ and that the $\gcd$ is defined up to sign. As convention, we take the positive $\gcd$, but this does not affect any further computations. Consequently,
	\begin{align*}
		\ker(\Lambda_0) \cong
		\begin{cases}
			\ZZ^2 & \text{if } |m|=|n| = 1; \\
			\ZZ   & \text{otherwise;}
		\end{cases}   \quad \text{and} \quad \coker(\Lambda_0) \cong \frac{\ZZ}{\gcd(1-|n|,1-|m|) \ZZ}	 \oplus \ZZ.
	\end{align*}

	Now, treat $[u_e]$ as the column $(\begin{smallmatrix}
			1 & 0
		\end{smallmatrix})^T$ and $[u_{\ol{e}}]$ as the column $(\begin{smallmatrix}
			0 & 1
		\end{smallmatrix})^T$. Then Proposition~\ref{prop:genbs} together with Lemma~\ref{lem:eebarkk} imply that the map $\Lambda_1\colon K_1(B) \to K_1(B)$ can be identified with left multiplication by the matrix
	\[
		\begin{pmatrix}
			1-\sgn(n)m &
			|n|-1        \\ |m|-1 & 1-\sgn(m)n
		\end{pmatrix}.
	\]
	We consider separately the cases where $mn > 0$ and where $mn < 0$.

	First suppose that $mn > 0$, in which case $\sgn(n)m = |m|$ and $\sgn(m)n = |n|$. A similar computation to that of $\Lambda_0$ shows that for $mn \ge 0$ we have,
	\[
		\ker(\Lambda_1) \cong \begin{cases}
			\ZZ^2 & \text{if } mn = 1; \\
			\ZZ   & \text{otherwise;}
		\end{cases} \quad \text{and} \quad \coker(\Lambda_1) \cong \frac{\ZZ}{\gcd(1-|n|,1-|m|) \ZZ}	 \oplus \ZZ.
	\]
	Now suppose that $mn <0$ so that $\sgn(m)n = -|n|$ and $\sgn(n)m = -|m|$. It follows from \cite[Theorem 2.4]{Sta16} that the matrix associated to $\Lambda_1$ has Smith normal form,
	\[
		\begin{pmatrix}
			\gcd(1+|m|,1-|m|,1+|n|,1-|n|) & 0 \\ 0 & \frac{2(|m|+|n|)}{\gcd(1+|m|,1-|m|,1+|n|,1-|n|) }
		\end{pmatrix}.
	\]
	Observe that
	\[
		\gcd(1+|m|,1-|m|,1+|n|,1-|n|) =
		\begin{cases}
			2 & \text{if } mn \text{ is odd}   \\
			1 & \text{if } mn \text{ is even,}
		\end{cases}
	\]
	so
	\begin{align*}
		\ker(\Lambda_1) = 0 \quad \text{and} \quad \coker(\Lambda_1) \cong
		\begin{cases}
			\frac{\ZZ}{2\ZZ} \oplus \frac{\ZZ}{(|m|+|n|)\ZZ} & \text{if } mn \text{ is odd}  \\
			\frac{\ZZ}{2(|m|+|n|)\ZZ}                        & \text{if } mn \text{ is even.}
		\end{cases}
	\end{align*}

	Since $K_i(B)$ is free abelian for $i=0,1$ it now follows from a standard argument with the six-term sequence of Theorem~\ref{thm:ktheory} that $K_i(C^*(\Gg)) \cong \ker(\Lambda_{1-i}) \oplus \coker(\Lambda_i)$. To summarise,
	\begin{align*}
		K_0(C^*(\Gg)) & \cong
		\begin{cases}
			\ZZ^4
			                                             & \text{if } mn =1;              \\
			\ZZ^2 \oplus\frac{\ZZ}{\gcd(1-|n|,1-|m|) \ZZ}
			                                             & \text{if } mn > 1;             \\
			\ZZ \oplus \frac{\ZZ}{\gcd(1-|n|,1-|m|) \ZZ} & \text{if } mn < 0; \text{ and}
		\end{cases} \\
		K_1(C^*(\Gg)) & \cong
		\begin{cases}
			\ZZ^4 & \text{if } mn = 1;                              \\
			\ZZ^2 \oplus \frac{\ZZ}{\gcd(1-|n|,1-|m|) \ZZ}
			      & \text{if } mn > 1;                              \\
			(\ZZ \oplus \frac{\ZZ}{2\ZZ})^2
			      & \text{if } mn = -1;                             \\
			\ZZ \oplus \frac{\ZZ}{2\ZZ} \oplus \frac{\ZZ}{(|m|+|n|)\ZZ}
			      & \text{if } mn < -1 \text{ is odd}; \text{ and } \\
			\ZZ \oplus \frac{\ZZ}{2(|m|+|n|)}
			      & \text{if } mn<-1 \text{ is even.}
		\end{cases}
	\end{align*}
	We now consider a few specific cases. If $m = n = 1$ then the universal covering tree $X_\Gg$ is an infinite $2$-regular tree with $2$ boundary points. Hence,
	\[
		C^*(\Gg) \cong C(\partial X_\Gg) \rtimes \pi_1(\Gg) \cong \CC^2 \rtimes_{\id} \ZZ^2 \cong \CC^2 \ox C^*(\ZZ^2) \cong C(\TT^2)^2.
	\]
	The $K$-groups for $C(\TT^2)^2$ agree with the $m=n=1$ case above. Similarly, if $m=1$ and $n=-1$, then
	\[
		C^*(\Gg) \cong C( \partial X_\Gg) \rtimes \pi_1(\Gg) \cong \CC^2 \rtimes_{\id} BS(1,-1) \cong \CC^2 \ox C^*(BS(1,-1)) \cong (C(\TT) \rtimes_{\gamma} \ZZ )^2,
	\]
	where $\gamma \in \Aut(C(\TT))$ is induced by conjugation on $\TT$. An application of the Pimsner-Voiculescu sequence for $C(\TT) \rtimes_{\gamma} \ZZ$ shows that the $K$-groups above when $m=1$ and $n=-1$ agree with the $K$-groups of $(C(\TT) \rtimes_{\gamma} \ZZ )^2$.

	Finally, we remark that \cite[Corollary 7.15]{BMPST17} states that $C^*(\Gg)$ is a Kirchberg algebra if and only if $|m|,|n| \ge 2$ and $|m| \ne |n|$ . Hence the preceding $K$-theory computations can be used to classify $C^*(\Gg) \cong \Oo_D$ in this case.
\end{example}

For a general locally finite nonsingular GBS graph of groups we have the following result.
\begin{thm}
	Let $\Gg = (\Gamma,G)$ be a locally finite nonsingular GBS graph of groups. Suppose that for each $e \in \Gamma^1$ the map $\alpha_e$ is given by multiplication by $m_e \in \ZZ$. Let $1_e$ denote the identity in $C^*(G_e) \cong C(\TT)$ and let $u_e$ denote the unitary $z \mapsto z$ in $C^*(G_e)$. Then
	\[
		K_0(C^*(\Gg)) \cong \coker(\Lambda_0) \oplus \ker(\Lambda_1) \quad \text{and} \quad K_1(C^*(\Gg)) \cong \coker(\Lambda_1) \oplus \ker(\Lambda_0),
	\]
	where $\Lambda_0 \colon \bigoplus_{e \in \Gamma^1} \ZZ[1_e] \to \bigoplus_{e \in \Gamma^1} \ZZ[1_e]$ and $\Lambda_1 \colon  \bigoplus_{e \in \Gamma^1} \ZZ[u_e] \to  \bigoplus_{e \in \Gamma^1} \ZZ[u_e]$ are given by
	\begin{align*}
		\Lambda_0([1_f]) & = [1_f] - (|m_{\ol{f}}|-1)[1_{\ol{f}}]\,  - \sum_{\substack{r(e) = s(f) \\ f \ne \ol{e}}}|m_e| [1_e]
		\quad \text{and}                                                                           \\
		\Lambda_1([u_f]) & = [u_f] - (|m_{\ol{f}}|-1)[u_{\ol{f}}]\, - \sum_{\substack{r(e) = s(f)  \\ f \ne \ol{e}}} \sgn(m_e) m_{\ol{f}} [u_e].
	\end{align*}
\end{thm}
\begin{proof}
	Since $K_0(B) = \bigoplus_{e \in \Gamma^1} \ZZ[1_e]$ and $K_1(B)  = \bigoplus_{e \in \Gamma^1} \ZZ[u_e]$ are free abelian it follows that $K_i(C^*(\Gg)) \cong \coker(\Lambda_i) \oplus \ker(\Lambda_{1-i})$. The descriptions of $\Lambda_0$ and $\Lambda_1$ follow immediately from Proposition~\ref{prop:genbs} together with Lemma~\ref{lem:eebarkk}.
\end{proof}

\section{Poincar\'e Duality}
\label{sec:poincare}

In this section we focus our attention on the case where $\Gg$ is the loop of groups considered in Example~\ref{ex:loop}. In this case Theorem~\ref{thm:bmpst17iso} asserts that $C^*(\Gg)$ is isomorphic to $C( \partial X_\Gg) \rtimes BS(n,m)$, where $X_\Gg$ is an $(|m|+|n|)$-regular tree acted on by the Baumslag-Solitar group $BS(n,m) = \langle a,t \mid ta^nt^{-1} = a^m \rangle$, \cite[p. 126]{BMPST17}.

We ask the question whether $C( \partial X_\Gg) \rtimes BS(n,m)$ satisfies
Poincar\'{e} duality in Kasparov theory. Poincar\'{e} duality for  hyperbolic groups
acting on their Gromov boundary has been proved by Emerson \cite{Eme03},
and similar results have been obtained for the Ruelle algebras of a quite general class of hyperbolic dynamical systems \cite{KPW17}.
The Baumslag-Solitar groups are quite different from the class of
hyperbolic groups, and so too is our method of proof.

Here we utilise the methods of \cite{RRS19}. There a systematic method
for deciding if fundamental classes in $K$-theory and $K$-homology
exist for a given Cuntz-Pimsner algebra whose coefficient algebra satisfies
Poncar\'e duality.

In our case the coefficient algebra will be (Morita equivalent to) a direct sum
of copies of $C^*(\ZZ)\cong C(\TT)$, which does indeed satisfy Poincar\'e duality.
Then \cite{RRS19} provides necessary and sufficient conditions on the dynamics
to be able to lift the $K$-theory fundamental class for $C(\TT)$ to a $K$-theory
fundamental class for $C^*(\Gg)$. The existence of a $K$-homology fundamental class
we leave open.

Nevertheless, even with `half' of the Poincar\'{e} duality,
the methods of \cite{RRS19} give isomorphisms between $K$-theory and $K$-homology.
Thus we obtain a computation of the $K$-homology for the $C^*$-algebras of the loops of groups from Example~\ref{ex:loop},
in spite of the large amount of torsion in the $K$-theory preventing the use of duality (i.e. index pairing)
methods.

Given an $A$--$B$-correspondence $(\phi,X)$, taking the conjugate module $X^*$
gives a $B$--$A$-module. We write $\flat(x)\in X^*$ for the copy of $x\in X$.
Denoting the opposite algebra of $A$ by $A^{op}$ and similarly for $B^{op}$, we may regard $X^*$ as an $A^{op}$--$B^{op}$-module $(\phi^{op},X^*)$ via
\begin{equation}
	\phi^{op}(a^{op})\flat(x)\,b^{op}:=\flat(\phi(a^*)xb^*),\quad a\in A,\ b\in B, x\in X.
	\label{eq:op-actions}
\end{equation}
The map $\flat\colon X\to X^*$ is anti-linear.
As shown in \cite[Section 2]{RRS19}, the left $B$-valued inner product
on $X^*$ can also be reinterpreted as a right $B^{op}$-valued inner product.
Thus the conjugate $(\phi^{op},X^*)$ of an $A$--$B$-correspondence is an $A^{op}$--$B^{op}$-correspondence.

Analogous to the conjugate module, one can define the conjugate algebra.
Given a $C^*$-algebra $A$, the conjugate algebra $\overline{A}$
has the same product and
adjoint, but the conjugate scalar multiplication, \cite[p. 157]{RW98}.
The map $a\mapsto a^*$ gives a (complex linear) isomorphism $\overline{A}\to A^{op}$. If $A$ is commutative we of course have $A\cong A^{op}$, and this proves
that if $A$ is a commutative $C^*$-algebra then $A\cong \overline{A}$.

\begin{lem} With  the convention $C(\TT)^0=\{0\}$ and defining
	\begin{equation}
		D=C(\TT)^{|n|}\oplus C(\TT)^{|m|}\oplus C(\TT)^{|n|-1}\oplus C(\TT)^{|m|-1},
		\label{eq:sum-triv}
	\end{equation}
	let $(\ol{\varphi},D)$ be the graph of groups correspondence over $ B=C(\TT)\oplus C(\TT)$ for the loop of groups from Example~\ref{ex:loop}, and let $(\ol{\varphi}^{op},D^{*op})$ be the conjugate correspondence.
	Then $D\cong D^{*op}$, where we use the fact that $C(\TT)^{op}=C(\TT)\cong \overline{C(\TT)}$.
\end{lem}
\begin{proof}
	As above, the conjugate correspondence $(\ol{\varphi}^{op},D^{*op})$ is
	a $C(\TT)$--$C(\TT)$-correspondence.
	We observe that the left action on $D$ is highly non-trivial, while the right action
	is just componentwise multiplication, and this remains true (up to a conjugation) for
	$D^{*op}$, as Equation \eqref{eq:op-actions} or \cite[p. 1121]{RRS19} shows.

	Rather than using the
	map $\flat:\,D\to D^{*op}$ we utilise the fact that $D$ is trivial (as a right module) to define
	$S:\,D\to D^{*op}$ to be $\flat\circ *$ where $*$ is
	the componentwise $C^*$-adjoint. Then
	$S$ is a linear map, and intertwines the actions of
	$C(\TT)\oplus C(\TT)$ since
	\begin{align*}
		S(\ol{\varphi}(a)xb)=\flat(b^*(x)_i^*\ol{\varphi}(a^*)_{ji})=\flat(\ol{\varphi}(a^*)_{ji}(x)_i^*b^*)
		=\ol{\varphi}^{op}(a^{op})\flat(x^*)b^{op}
	\end{align*}
	using the commutativity of $C(\TT)$.
	That $S$ is a bijection is clear.
\end{proof}
The same argument shows that $[D\ox A^{op}]=[A\ox D^{*op}]\in KK(\CC,A\ox A^{op})$.
\begin{thm}
	\label{thm:PD}
	Consider a loop of groups $\Gg$ as in Example \ref{ex:loop}, so
	that $C^*(\Gg) \cong C(\partial X_{\Gg})\rtimes BS(m,n)$ where $BS(m,n)$ is a Baumslag-Solitar group. Then there exists
	a $K$-theory class $\delta\in KK(\CC,C^*(\Gg)\ox C^*(\Gg))$ such that
	the Kasparov product with $\delta$
	\[
		\delta\ox_{C^*(\Gg)}\cdot \,:\,K^*(C^*(\Gg))\to K_*(C^*(\Gg))
	\]
	is an isomorphism.
\end{thm}

\begin{proof} First of all $C^*(\Gg)\cong C^*(\Gg)^{op}$ via the inverse map of the transformation groupoid underlying the crossed product $C(\partial X_\Gg) \rtimes BS(n,m)$,  \cite[Theorem 2.1]{BS17}.
	The algebra $B$ for these examples is $C(\TT)\oplus C(\TT)$, which
	satisfies Poincar\'{e} duality, \cite{Kas88}, \cite[Lemma 3.5]{RRS19}, having  fundamental class the direct sum  of the fundamental classes for the summands.
	Let $[z] \in KK^1(\CC, C(\TT))$ denote the class of the unitary $z \mapsto z$, and let $[\iota_{\CC,C(\TT)}] \in KK^1(\CC, C(\TT))$ denote the class of the unital inclusion of $\CC$ into $C(\TT)$.
	The $K$-theory fundamental class
	for the circle is $\beta=[z]\ox[\iota_{\CC,C(\TT)}]-[\iota_{\CC,C(\TT)}]\ox[z]\in KK^1(\CC,C(\TT)\ox C(\TT))$, and we let $\beta^2=\beta\oplus\beta$ be the corresponding class for
	$B=C(\TT)\oplus C(\TT)$.

	We summarise the argument of \cite[Theorem 4.8]{RRS19}.
	Using $[D]=[D^{*op}]$, $[D\ox B]=[B\ox D^{*op}]$ and the  definition
	\[
		\beta^2\ox_{B}[D]:=\beta^2\ox_{B\ox B}[D\ox C(\TT)],\qquad
		\beta^2\ox_{B}[D^{*op}]:=\beta^2\ox_{B\ox B}[C(\TT)\ox D^{*op}],
	\]
	we see that $\beta^2\ox_{B}[D]=\beta\ox_{B}[D^{*op}]$. Hence,
	with $\iota:\,B\to C^*(\Gg)$ the inclusion and  $\partial^*$, $\partial_*$ denoting
	boundary maps,
	the diagram
	\[
		\xymatrix{ \cdots\ar[r]&K^0(B)\ar[r]^{1-D}\ar[d]_{\beta^2\ox}&K^0(B)\ar[r]^{\partial^*}\ar[d]_{\beta^2\ox}&K^1(C^*(\Gg))\ar[r]^{\iota^*}\ar[d]_{\delta\ox}&K^1(B)\ar[r]^{1-D}\ar[d]_{\beta^2\ox}&K^1(B)\ar[r]^{\partial^*}\ar[d]_{\beta^2\ox}&\cdots\\
			\cdots\ar[r]&K_1(B)\ar[r]^{1-D}&K_1(B)\ar[r]^{\iota_*}&K_1(C^*(\Gg))\ar[r]^{\partial_*}&K_0(B)\ar[r]^{1-D}&K_0(B)\ar[r]^{\iota_*}&\cdots
		}
	\]
	commutes if and only if
	\begin{equation}
		(\delta\ox\cdot)\circ \partial^*=\iota_*\circ (\beta^2\ox\cdot)\quad \mbox{and}\quad
		\partial_*\circ(\delta\ox\cdot)=(\beta^2\ox\cdot)\circ\iota^*.
		\label{eq:bob}
	\end{equation}
	As the Kasparov product $\beta^2\ox_B\cdot:\,K^j(B)\to K_{j+1}(B)$
	is an isomorphism for $j=0,1$, the five lemma tells us that if if these
	two conditions on $\delta$ are satisfied, the Kasparov product with $\delta$ will provide an isomorphism.


	Finally, since $\beta^2\ox_B[D]=\beta^2\ox_{B}[D^{*op}]$, a class $\delta$ satisfying
	the two conditions \eqref{eq:bob} can be explicitly constructed. The recipe is given in
	\cite[Lemma 4.7]{RRS19}.
\end{proof}

\begin{cor}
	The $K$-homology of $B(m,n)$ is given by
	\begin{align*}
		K^0(C^*(\Gg)) & \cong
		\begin{cases}
			\ZZ^4
			                                             & \text{if } mn =1;              \\
			\ZZ^2 \oplus\frac{\ZZ}{\gcd(1-|n|,1-|m|) \ZZ}
			                                             & \text{if } mn > 1;             \\
			\ZZ \oplus \frac{\ZZ}{\gcd(1-|n|,1-|m|) \ZZ} & \text{if } mn < 0; \text{ and}
		\end{cases} \\
		K^1(C^*(\Gg)) & \cong
		\begin{cases}
			\ZZ^4 & \text{if } mn = 1;                             \\
			\ZZ^2 \oplus \frac{\ZZ}{\gcd(1-|n|,1-|m|) \ZZ}
			      & \text{if } mn > 1;                             \\
			(\ZZ \oplus \frac{\ZZ}{2\ZZ})^2
			      & \text{if } mn = -1;                            \\
			\ZZ \oplus \frac{\ZZ}{2\ZZ} \oplus \frac{\ZZ}{(|m|+|n|)\ZZ}
			      & \text{if } mn < -1 \text{ is odd}; \text{ and} \\
			\ZZ \oplus \frac{\ZZ}{2(|m|+|n|)}
			      & \text{if } mn<-1 \text{ is even.}
		\end{cases}
	\end{align*}
\end{cor}
\begin{proof}
	The class $\delta$ provides an isomorphism
	$K^j(C^*(\Gg))\cong K_j(C^*(\Gg^{op}))\cong K_j(C^*(\Gg^{op}))$.
\end{proof}
No method employing duality (index) pairings between $K$-theory and $K$-homology can obtain this result when $mn\neq 1$ due to the torsion subgroups: see for instance \cite[Chapter 7]{HR00}.


%

\bigskip{\footnotesize%
	\textsc{School of Mathematics and Applied Statistics,
		University of Wollongong, Wollongong
		NSW  2522, Australia} \par
	\textit{E-mail addresses} A.~Mundey: \texttt{amundey@uow.edu.au} \\
	\hphantom{E-mail addresses}\quad\ \,\,A.~Rennie: \texttt{renniea@uow.edu.au}
}
\end{document}